\documentclass[a4paper,11pt,twoside]{article}

\usepackage{geometry}
\usepackage[utf8]{inputenc}
\usepackage{lmodern}
\usepackage[T1]{fontenc}
\usepackage{amsmath,amssymb,amsthm,empheq,cases}
\usepackage{hyperref}
\hypersetup{colorlinks,
            citecolor=red, 
            filecolor=black,
            linkcolor=blue,
            urlcolor=black}
\usepackage{tikz}
\usepackage{graphicx}
\usepackage{url}

\def \dis {\displaystyle}

\def \NN {\mathbb N}
\def \ZZ {\mathbb Z}

\def \RR {\mathbb R}
\def \CC {\mathbb C}

\def \A {\mathcal{A}}

\def \E {\mathcal{E}}
\def \F {\mathcal{F}}

\def \L {\mathcal{L}}
\def \M {\mathcal{M}}

\def \P {\mathcal{P}}

\def \ecart {\noalign{\medskip}}

\theoremstyle{definition}
\newtheorem{Th}{Theorem}[section]
\newtheorem{Prop}[Th]{Proposition}
\newtheorem{Lem}[Th]{Lemma}
\newtheorem{Cor}[Th]{Corollary}
\newtheorem{Def}[Th]{Definition}
\newtheorem{Rem}[Th]{Remark}
 
\def \refS #1{Section~\ref{#1}}
\def \refD #1{Definition~\ref{#1}}
\def \refT #1{Theorem~\ref{#1}}
\def \refL #1{Lemma~\ref{#1}}
\def \refC #1{Corollary~\ref{#1}}
\def \refP #1{Proposition~\ref{#1}}
\def \refR #1{Remark~\ref{#1}}

\title{Generalized diffusion problems in a conical domain, part II}
\author{Rabah Labbas, St\'ephane Maingot \& Alexandre Thorel \\ \ecart
\scriptsize R. L., S. M. \& A. T.: Normandie Univ, UNIHAVRE, LMAH, FR-CNRS-3335, ISCN, 76600 Le Havre, France. \\ \ecart 
\scriptsize rabah.labbas@univ-lehavre.fr, stephane.maingot@univ-lehavre.fr, alexandre.thorel@univ-lehavre.fr}
\date{}

\begin{document}

\maketitle

\begin{abstract}

After different variables and functions changes, the generalized dispersal problem, recalled in \eqref{Pb cone fini} below and considered in part I, see \cite{Cone P1},  leads us to invert a sum of linear operators in a suitable Banach space, see \eqref{eq} below.   

The essential result of this second part lies in the complete study of this sum using the two well-known strategies: the one of Da Prato-Grisvard \cite{daprato-grisvard} and the one of Dore-Venni \cite{dore-venni}.
\end{abstract}

\section{Introduction and main result}\label{Sect Intro}

This work is a natural continuation of \cite{Cone P1}, where we have studied, for $k>0$, the following problem 
\begin{equation}\label{Pb cone fini}
\left\{ \begin{array}{ll}
\Delta ^{2}u-k\Delta u=f &\text{in }S_{\omega,\rho} \\ \ecart 
u=\dfrac{\partial u}{\partial n}=0 & \text{on }\Gamma _{0}\cup \Gamma_{\omega,\rho},
\end{array}\right. 
\end{equation}
where, for given $\rho>0$ and $\omega \in (0,2\pi]$:
\begin{equation*}
\left\{ 
\begin{array}{lll}
S_{\omega,\rho} & = &\dis \left\{ (x,y)=(r\cos \theta ,r\sin \theta ):0<r<\rho \text{ and } 0<\theta <\omega \right\}, \quad \\
\Gamma _{0} & = &\dis (0,+\infty) \times \left\{ 0\right\}  \\ 
\Gamma _{\omega,\rho} & = & \dis \left\{ (r\cos \omega ,r\sin \omega ): 0<r<\rho\right\}.
\end{array}\right. 
\end{equation*} 
Note that a similar problem set in a cylindrical domain has been already studied by \cite{LMT}, but in the present paper, since the domain is conical, the study is completely different.
 
After using the polar coordinates $v(r,\theta )=u(r\cos \theta ,r\sin \theta )$, the function
$$\phi(t)(\theta) := \phi (t,\theta )= \frac{v(\rho e^{-t},\theta )}{\rho e^{-t}}, \quad t > 0,$$
and the abstract vector-valued function
\begin{equation*}
V(t) = \left( 
\begin{array}{c}
e^{\nu t}\phi (t) \\ 
e^{\nu t} \phi''(t)
\end{array}\right), \quad \text{with} \quad \nu = 3 - \frac{2}{p} \in (1,3),
\end{equation*}
the above study led us, to solve the following abstract equation 
\begin{equation}\label{eq}
\left(\L_1 + \L_2\right) V + k \rho^2 \left(\P_1 + \P_2 \right)V = \F,
\end{equation}
where $\F \in L^p(0,+\infty;W_{0}^{2,p}(0,\omega) \times L^{p}(0,\omega))$, with $p\in(1,+\infty)$. Operators $\L_1$, $\L_2$, $\P_1$ and $\P_2$, introduced in \cite{Cone P1}, are recalled here for the reader convenience:
\begin{equation*}
\left\{ \begin{array}{cll}
D(\mathcal{L}_{1}) & = & \dis \left\{ V\in W^{2,p}(0,+\infty ;X): V(0) = V(+\infty) = 0\right\} \\ \ecart
\left[ \mathcal{L}_{1}(V)\right] (t) & = & \dis \left( \partial_{t}-\nu I\right)^{2}V(t)=V''(t)-2\nu V'(t)+\nu^{2}V(t),
\end{array}\right.
\end{equation*}
with $\nu \in \RR$ fixed, 
\begin{equation*}
\left\{ 
\begin{array}{lll}
D(\mathcal{L}_{2}) & = &\dis \left\{ V\in L^{p}(0,+\infty ;X): \text{for }a.e.~t\in
(0,+\infty ),~ V(t)\in D(\mathcal{A})\right\} \\ \ecart
\left[ \mathcal{L}_{2}(V)\right] (t) & = & -\mathcal{A}V(t),
\end{array}
\right.
\end{equation*}
with
\begin{equation*}
\left\{ 
\begin{array}{lll}
D(\mathcal{A}) & = & \dis \left[ W^{4,p}(0,\omega)\cap W_{0}^{2,p}(0,\omega)\right] \times W_{0}^{2,p}(0,\omega)\subset X \\ \ecart
\mathcal{A}\left( 
\begin{array}{c}
\psi _{1} \\ 
\psi _{2}
\end{array}
\right) &=& \left( 
\begin{array}{c}
\psi _{2} \\ 
-\left( \dfrac{\partial ^{2}}{\partial \theta ^{2}}+1\right) ^{2}\psi
_{1}-2\left( \dfrac{\partial ^{2}}{\partial \theta ^{2}}-1\right) \psi _{2}%
\end{array}
\right), \quad \left( 
\begin{array}{c}
\psi _{1} \\ 
\psi _{2}
\end{array}
\right) \in D(\A),
\end{array}
\right.
\end{equation*}
\begin{equation*}
\left\{ 
\begin{array}{lll}
D(\mathcal{P}_{1}) & = &\dis \left\{ V\in L^{p}(0,+\infty ;X): \text{for }a.e.~t\in
(0,+\infty ),~ V(t)\in D(\mathcal{A}_0)\right\} \\ \ecart
\left[ \mathcal{P}_{1}(V)\right] (t) & = & -e^{-2t} \mathcal{A}_{0}V(t),
\end{array}
\right.
\end{equation*}
with
\begin{equation*}
\left\{ 
\begin{array}{lll}
D(\mathcal{A}_{0}) & = & W_{0}^{2,p}(0,\omega) \times L^{p}(0,\omega) = X \\ \ecart
\mathcal{A}_{0}\left( 
\begin{array}{c}
\psi _{1} \\ 
\psi _{2}
\end{array}
\right) & = & \dis \left( 
\begin{array}{c}
0 \\ 
\left( \dfrac{\partial ^{2}}{\partial \theta ^{2}}+1\right) \psi _{1}+\psi_{2}
\end{array}
\right), \quad \left( 
\begin{array}{c}
\psi _{1} \\ 
\psi _{2}
\end{array}
\right) \in D(\A_0),
\end{array}\right.
\end{equation*}
and
\begin{equation*}
\left\{ 
\begin{array}{lll}
D(\mathcal{P}_{2}) & = &\dis W^{1,p}(0,+\infty ;X) \\ \ecart
\left[ \mathcal{P}_{2}(V)\right] (t) & = & -e^{-2t}\left( \mathcal{B}_{2}V\right) (t),
\end{array}\right.
\end{equation*}
with
\begin{equation*}
\mathcal{B}_{2}=\left(\begin{array}{cc}
0 & 0 \\ 
-2(\partial_{t}-\nu I) & 0
\end{array}\right).
\end{equation*}
In the present paper, we will focus ourselves on the resolution of equation \eqref{eq}. Among others, we need to use the fact that the roots of the following equation
$$\left(\sinh(z) + z\right)\left(\sinh(z) - z\right) = 0,$$
in $\CC_+:=\{w \in \CC : \text{Re}(w) > 0\}$, constitute a family of complex numbers $(z_j)_{j\geqslant 1}$ such that 
\begin{equation*}
\tau := \min_{j\geqslant 1}\left|\text{Im}(z_j)\right| > 0  \quad \text{and} \quad |z_j| \longrightarrow + \infty.
\end{equation*}
These roots are computed in \cite{fadle} with $\tau \simeq 4.21239$.

Our main result is the following.
\begin{Th}\label{Th principal}
Let $\F \in L^p(0,+\infty;X)$ and assume that 
\begin{equation}\label{hyp inv sum}
\omega\nu < \tau.
\end{equation}
Then, there exists $\rho_0 > 0$ such that for all $\rho \in (0,\rho_0]$, the abstract equation 
\begin{equation*}
\left(\L_1 + \L_2\right) V + k \rho^2 \left(\P_1 + \P_2 \right)V = \F,
\end{equation*}
has a unique classical solution $V \in L^p(0,+\infty;X)$, that is 
$$V \in W^{2,p}(0,+\infty;X)\cap L^p(0,+\infty;D(\A)).$$
In particular, $\L_1 + \L_2$ is a closed operator and $V \in D(\L_1 + \L_2)$.
\end{Th}

This second part is organized as follows. \refS{Sect Def} is devoted to some recalls. In \refS{Sect Spectral study of L1 and L2}, we analyze the spectral properties of operators $\L_1$ and $\L_2$ in view to study the invertibility of the closedness of the sum $\overline{\L_1+\L_2}$ in \refS{Sect study of L1+L2}. In \refS{Sect Back to Abstract Pb}, by considering that operator $k\rho^2\left(\P_1+\P_2\right)$ is a perturbation, we deduce the existence and the uniqueness of a strong solution of equation \eqref{eq}. Finally, \refS{Sect Proof of main Th} is devoted to the proof of our main result given in \refT{Th principal}.

\section{Definitions and prerequisites} \label{Sect Def}

\subsection{The class of Bounded Imaginary Powers of operators}

\begin{Def}\label{Def Banach}
A Banach space $X$ is a UMD space if and only if for all $p\in(1,+\infty)$, the Hilbert transform is bounded from $L^p(\RR,X)$ into itself (see \cite{bourgain} and \cite{burkholder}).
\end{Def}
\begin{Def}\label{Def op Sect}
Let $\alpha \in(0,\pi)$. Sect($\alpha$) denotes the space of closed linear operators $T_1$ which satisfying
$$
\begin{array}{l}
i)\quad \sigma(T_1)\subset \overline{S_{\alpha}},\\ \ecart
ii)\quad\forall~\alpha'\in (\alpha,\pi),\quad \sup\left\{\|\lambda(\lambda\,I-T_1)^{-1}\|_{\L(X)} : ~\lambda\in\CC\setminus\overline{S_{\alpha'}}\right\}<+\infty,
\end{array}
$$
where
\begin{equation}\label{defsector}
S_\alpha\;:=\;\left\{\begin{array}{lll}
\left\{ z \in \CC : z \neq 0 ~~\text{and}~~ |\arg(z)| < \alpha \right\} & \text{if} & \alpha \in (0, \pi] \\ \ecart
\,(0,+\infty) & \text{if} & \alpha = 0,
\end{array}\right.
\end{equation} 
see \cite{haase}, p. 19. Such an operator $T_1$ is called sectorial operator of angle $\alpha$.
\end{Def}
\begin{Rem}
From \cite{komatsu}, p. 342, we know that any injective sectorial operator~$T_1$ admits imaginary powers $T_1^{is}$ for all $s\in\RR$; but in general, $T_1^{is}$ is not bounded.
\end{Rem} 
\begin{Def}\label{Def BIP}
Let $\theta \in [0, \pi)$. We denote by BIP$(X,\theta)$, the class of sectorial injective operators $T_2$ such that
\begin{itemize}
\item[] $i) \quad ~~\overline{D(T_2)} = \overline{R(T_2)} = X,$

\item[] $ii) \quad ~\forall~ s \in \RR, \quad T_2^{is} \in \L(X),$

\item[] $iii) \quad \exists~ C \geqslant 1 ,~ \forall~ s \in \RR, \quad ||T_2^{is}||_{\L(X)} \leqslant C e^{|s|\theta}$,
\end{itemize}
see~\cite{pruss-sohr}, p. 430.
\end{Def}

\subsection{Recall on the sum of linear operators}\label{sect M1 M2}

Let us fix a pair of two closed linear densely defined operators $\mathcal{M}_{1}$ and $\mathcal{M}_{2}$ in a general Banach space $\mathcal{E}.$ We note their domains by $D(\mathcal{M}_{1})$ and $D(\mathcal{M}_{2})$ respectively. Then we can define their sum by 
\begin{equation*}
\left\{ \begin{array}{l}
\mathcal{M}_{1}w+\mathcal{M}_{2}w \\ 
w\in D(\mathcal{M}_{1})\cap D(\mathcal{M}_{2}).
\end{array}\right. 
\end{equation*}
We assume the following hypotheses
\begin{itemize}

\item[$(H_1)$] There exist $\theta_{\M_1} \in [0,\pi)$, $\theta
_{\M_{2}}\in [0,\pi)$, $C > 0$ and $R > 0$ such that  
\begin{equation*}
\left\{ 
\begin{array}{l}
\rho \left( \mathcal{M}_{1}\right) \supset \Sigma _{1,R}=\left\{ z\in \CC\setminus \left\{ 0\right\} : |z| \geqslant R \text{ and }\left\vert \arg (z)\right\vert <\pi -\theta
_{\M_{1}}\right\}  \\ 
\forall \,z\in \Sigma_{1,R}, \quad \left\Vert \left( \mathcal{M}_{1}-zI\right) ^{-1}\right\Vert \leqslant \dfrac{C}{\left\vert z\right\vert},
\end{array}\right. 
\end{equation*}
and
\begin{equation*}
\left\{ 
\begin{array}{l}
\rho \left( \mathcal{M}_{2}\right) \supset \Sigma _{2,R}=\left\{ z\in \CC\setminus \left\{ 0\right\} : |z| \geqslant R \text{ and }\left\vert \arg (z)\right\vert <\pi -\theta_{\M_{2}}\right\}  \\ 
\forall \,z\in \Sigma_{2,R}, \quad \left\Vert \left( \mathcal{M}_{2}-zI\right) ^{-1}\right\Vert \leqslant \dfrac{C}{\left\vert z\right\vert},
\end{array}\right.  
\end{equation*}
with
\begin{equation*}
\theta _{\M_{1}}+\theta _{\M_{2}}<\pi.
\end{equation*}

\item[$(H_2)$] $\sigma(\M_1) \cap \sigma(-\M_2) = \emptyset$.

\item[$(H_3)$] The resolvents of $\mathcal{M}_{1}$ and $\mathcal{M}_{2}$
commute, that is
\begin{equation*}
\left( \mathcal{M}_{1}-\lambda I\right) ^{-1}\left( \mathcal{M}_{2}-\mu
I\right) ^{-1}=\left( \mathcal{M}_{2}-\mu I\right) ^{-1}\left( \mathcal{M}_{1}-\lambda I\right) ^{-1},
\end{equation*}
for all $\lambda \in \rho \left( \mathcal{M}_{1}\right) $ and all $\mu \in
\rho \left( \mathcal{M}_{2}\right)$.
\end{itemize}

\begin{Rem}
Note that from $(H_2)$, we have $\rho \left(\mathcal{M}_{1}\right) \cup \rho \left(- \mathcal{M}_{2}\right) = \CC$ and in particular $\M_1$ or $\M_2$ is boundedly invertible.  
\end{Rem}

\begin{Th}[\cite{daprato-grisvard}, \cite{grisvard 2}] \label{Th daprato-grisvard}
Assume that $(H_1)$, $(H_2)$ and $(H_3)$ hold. Then, operator $\M_{1}+\M_{2}$ is closable. Its closure $\overline{\mathcal{M}_{1}+\mathcal{M}_{2}}$ is boundedly invertible and 
\begin{equation}\label{int M1+M2}
\left( \overline{\mathcal{M}_{1}+\mathcal{M}_{2}}\right)^{-1}=\frac{-1}{2i \pi }\int\limits_{\Gamma }\left( \mathcal{M}_{1}-zI\right)^{-1}\left( \mathcal{M}_{2}+zI\right) ^{-1}dz;
\end{equation}
where $\Gamma$ is a path which separates $\sigma \left(\mathcal{M}_{1}\right)$ and $\sigma \left(-\mathcal{M}_{2}\right)$ and joins $\infty e^{-i\theta _{0}}$ to $\infty e^{i\theta _{0}}$ with $\theta _{0}$ such that
\begin{equation*}
\theta _{\M_{1}}<\theta _{_{0}}<\pi -\theta _{\M_{2}}.
\end{equation*}
\end{Th}
This Theorem is proved in \cite{daprato-grisvard} (Theorem 3.7, p. 324), when $R = 0$ and has been extended to the case $R\geqslant 0$ in \cite{grisvard 2} (Theorem 2.1, p. 7). In this last case, the curve $\Gamma$ does not need to be connected. 

\begin{Cor} \label{Cor inv fermeture somme}
Assume $(H_1)$, $(H_2)$ and $(H_3)$ hold. Let $i = 1,2$ and $\E_i$ a Banach space with $D(\M_i) \hookrightarrow \E_i \hookrightarrow \mathcal{E}$ such that there exist $C>0$, $\delta \in (0,1)$ satisfying 
\begin{equation}\label{ineg norme w}
\left\{ 
\begin{array}{l}
\left\Vert w\right\Vert _{\mathcal{E}_i}\leqslant C\left( \left\Vert
w\right\Vert _{\mathcal{E}}+\left\Vert w\right\Vert _{\mathcal{E}}^{1-\delta}\left\Vert \mathcal{M}_{i}w\right\Vert _{\mathcal{E}}^{\delta }\right)  \\ 
\text{for every }w\in D(\mathcal{M}_{i}),
\end{array}\right. 
\end{equation}
then $D\left( \overline{\mathcal{M}_{1}+\mathcal{M}_{2}}\right) \subset \E_i.$
\end{Cor}

\begin{proof}
It is enough to prove that the integral in \eqref{int M1+M2} converges in $\E_1$. For all $\xi \in \E$, we have
$$\left\|\int\limits_{\Gamma }\left( \mathcal{M}_{1}-zI\right)^{-1}\left( \mathcal{M}_{2}+zI\right) ^{-1} \xi dz\right\|_{\E_1} \leqslant \int\limits_{\Gamma }\left\|\left( \mathcal{M}_{1}-zI\right)^{-1} \left( \mathcal{M}_{2}+zI\right) ^{-1}\xi\right\|_{\E_1} |dz|,$$
then, applying \eqref{ineg norme w}, we obtain
$$\begin{array}{l}
\dis \left\|\left( \mathcal{M}_{1}-zI\right)^{-1} \left( \mathcal{M}_{2}+zI\right) ^{-1}\xi\right\|_{\E_1} \\ \ecart
\leqslant \dis C \left\|\left( \mathcal{M}_{1}-zI\right)^{-1} \left( \mathcal{M}_{2}+zI\right) ^{-1}\xi\right\|_{\E} \\ \ecart
\dis + C \left\|\left( \mathcal{M}_{1}-zI\right)^{-1} \left( \mathcal{M}_{2}+zI\right) ^{-1}\xi\right\|_{\E}^{1-\delta} \left\|\M_1 \left( \mathcal{M}_{1}-zI\right)^{-1} \left( \mathcal{M}_{2}+zI\right) ^{-1}\xi\right\|_{\E}^{\delta};
\end{array}$$
Now, for all $z \in \Gamma$, we have
$$\begin{array}{l}
\left\|\left( \mathcal{M}_{1}-zI\right)^{-1} \left( \mathcal{M}_{2}+zI\right) ^{-1}\xi\right\|_{\E}^{1-\delta} \left\|\M_1 \left( \mathcal{M}_{1}-zI\right)^{-1} \left( \mathcal{M}_{2}+zI\right) ^{-1}\xi\right\|_{\E}^{\delta} \\ \ecart
\leqslant \dfrac{\left(C_1(\theta_1) C_2(\theta_2)\right)^{1-\delta}}{|z|^{2(1-\delta)}} \dfrac{C_1(\theta_1)^\delta C_2(\theta_2)^\delta}{|z|^\delta} \|\xi\|_{\E} = \dfrac{C_1(\theta_1) C_2(\theta_2)}{|z|^{1+(1-\delta)}} \|\xi\|_{\E},
\end{array}$$
from which we deduce the convergence of the integral in \eqref{int M1+M2}. The same result holds true replacing $\M_1$ by $\M_2$.
\end{proof}

\section{Spectral study of operators}\label{Sect Spectral study of L1 and L2}

In all the sequel, in view to apply the above results, we will consider the following particular Banach space 
$$\mathcal{E}=L^p(0,+\infty;X),$$
equipped with its natural norm.

\subsection{Study of operator $\L_1$}

We study the spectral equation 
\begin{equation*}
\mathcal{L}_{1} V - \lambda V = R,
\end{equation*} 
where $V\in D(\mathcal{L}_{1})$, $R \in \mathcal{E}$ and $\lambda \in \CC$ (which will be precised below), that is
\begin{equation}\label{Pb Amine}
\left\{ 
\begin{array}{l}
V''(t)-2\nu V'(t)+(\nu ^{2}-\lambda) V(t) = R(t), \quad t>0 \\ \ecart
V(0)=0,  ~~V(+\infty) = 0.
\end{array}
\right.
\end{equation}
We set 
$$\Sigma_{\nu} = \{z \in \CC \setminus \RR_-: \text{Re}(\sqrt{z}) > \nu \}.$$
Now, let us precise this set. For all $z = x + i y \in \CC \setminus \RR_-$, we have
$$\text{Re}(\sqrt{z}) > \nu  \Longleftrightarrow \sqrt{\frac{|z| + \text{Re}(z)}{2}} > \nu\Longleftrightarrow \sqrt{x^2 + y^2} > 2 \nu^2 - x.$$ 
\begin{itemize}
\item First case : if $x > \nu^2$, we have $\sqrt{x^2 + y^2} + x \geqslant  2x >  2\nu^2,$ then Re$(\sqrt{z}) > \nu$.

\item Second case : if $x \leqslant \nu^2$, then $y^2 + 4 \nu^2 x - 4 \nu^4 > 0$. Thus, we deduce that $\Sigma_{\nu}$ is strictly outside the parabola of equation
$$y^2 + 4 \nu^2 x - 4 \nu^4 = 0,$$
turned towards the negative real axis and passing through the points $(\nu^2,0)$, $(0,2\nu^2)$ and $(0,-2\nu^2)$.
\end{itemize}

\begin{figure}[ht]
  \begin{center}
  \begin{tikzpicture}[scale=0.2]
	\filldraw [fill=gray!20, line width=1, domain=-10:4, samples=500] plot(\x,{sqrt(64-16*\x)})--(4,0)--(-10,0);
	\filldraw [fill=gray!20, line width=1, domain=-10:4, samples=500] plot(\x,{-sqrt(64-16*\x)})--(4,0)--(-10,0);
		\draw [line width=1,->,>=latex] (-15,0) -- (15,0) node[pos=1,below] {$x$};
	\draw [line width=1,->,>=latex] (0,-15) -- (0,15) node[pos=1,right] {$y$};
	\node at (-0.7,-0.9) {$0$};
	\node at (5.2,-1.2) {$\nu^2$};
	\node at (1.8,9) {$2\nu^2$};
	\node at (2.2,-8) {-$2\nu^2$};
	\node at (10,7) {$\Sigma_\nu$};
  \end{tikzpicture}
  \end{center}
\caption{This figure represents $\Sigma_\nu$.}
\end{figure}
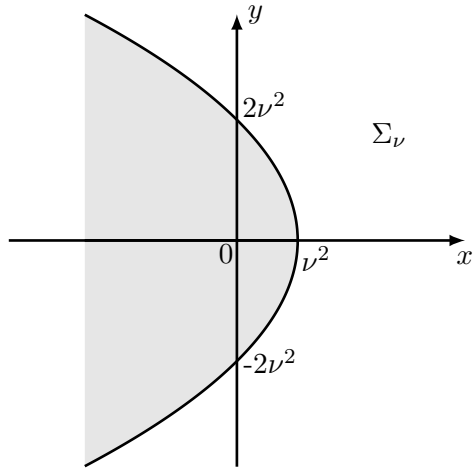

Now, let $\varepsilon_{\L_1}$ be a small fixed positive number and consider the following set
\begin{equation}\label{lambda estimation resolvante}
\Sigma_{\L_1} := \left\{\lambda \in \Sigma_\nu, \quad |\arg(\lambda)| \leqslant \pi - 2\varepsilon_{\L_1} \quad \text{and} \quad |\lambda| \geqslant \frac{4\nu^2}{\sin^2\left(\varepsilon_{\L_1}\right)}\right\}.
\end{equation}
We then obtain the following proposition.
\begin{Prop}\label{Prop L1}
The linear operator $\mathcal{L}_{1}$ is closed and densely defined in $W^{2,p}(0,+\infty;X)$. Moreover, there exists a constant $M_{\L_1} > 0$ such that for all $\lambda \in \Sigma_{\L_1}$, operator $\L_1 - \lambda I$ is invertible with bounded inverse and 
\begin{equation*}
\left\Vert \left(\L_1 - \lambda I\right)^{-1}\right\Vert_{\L(\mathcal{E})} \leqslant \frac{M_{\L_1}}{ |\lambda|}.
\end{equation*} 
Therefore, assumption $(H_1)$ in \refS{sect M1 M2} is verified for $\L_1$ with
\begin{equation}\label{theta L1}
\theta_{\L_1} = 2 \varepsilon_{\L_1}.
\end{equation}
\end{Prop}
\begin{proof}
Let $\lambda \in \Sigma_{\nu}$. From \cite{amine}, Theorem 2, p. 712, there exists a unique solution \mbox{$V \in W^{2,p}(0,+\infty;X)$} of problem \eqref{Pb Amine}, given by 
\begin{equation}\label{Rep V}
\begin{array}{rcl}
V(t) & = & \dis \frac{e^{t (\nu-\sqrt{\lambda})}}{2\sqrt{\lambda}} \int_0^{+\infty} e^{-s (\nu+\sqrt{\lambda})} R(s)~ds \\ \ecart
&& \dis - \frac{1}{2\sqrt{\lambda}} \left(\int_0^t e^{(t-s)(\nu-\sqrt{\lambda})} R(s)~ds + \int_t^{+\infty} e^{-(s-t)(\nu+\sqrt{\lambda})} R(s)~ds\right),
\end{array}
\end{equation}
see formula (15) in \cite{amine} where $L_1:= - \nu I - \sqrt{\lambda} I$ and $L_2:= \nu I - \sqrt{\lambda} I$. It follows that $\Sigma_{\nu} \subset \rho(\L_1)$. This proves that $\L_1$ is closed. The boundary conditions are verified by using Lemma~8, p. 718 in \cite{amine}. 

Moreover, for all $\lambda \in \Sigma_{\nu}$, since $\text{Re}(\sqrt{\lambda}) > \nu$, from \eqref{Rep V}, we obtain
$$\begin{array}{rcl}
\|V\|_{\mathcal{E}} &\hspace{-0.025cm} \leqslant &\hspace{-0.025cm} \dis \frac{1}{2\sqrt{|\lambda|}}\left(\int_0^{+\infty} e^{-t p(\text{Re}(\sqrt{\lambda})-\nu)} ~dt\right)^{1/p} \int_0^{+\infty} e^{-s (\nu+\text{Re}(\sqrt{\lambda}))} \|R(s)\|_{X}~ds \\ \ecart
&\hspace{-0.025cm}&\hspace{-0.025cm} \dis +  \sup_{t \in \RR_+} \left(\int_0^t \left|e^{(t-s)(\nu-\sqrt{\lambda})}\right|ds + \int_t^{+\infty} \left|e^{-(s-t)(\nu+\sqrt{\lambda})}\right| ds\right)\frac{\|R\|_{\mathcal{E}}}{2\sqrt{|\lambda|}},
\end{array}$$
hence, noting $q$ the conjugate exponent of $p$, we have
$$\begin{array}{rcl}
\|V\|_{\mathcal{E}} & \leqslant & \dis \left(\frac{1}{p (\text{Re}(\sqrt{\lambda})-\nu)} \right)^{1/p} \left(\int_0^{+\infty} e^{-s q(\nu+\text{Re}(\sqrt{\lambda}))}~ds\right)^{1/q} \frac{\|R\|_{\mathcal{E}}}{2\sqrt{|\lambda|}} \\ \ecart
&& \dis + \sup_{t \in \RR_+} \left(\frac{1 - e^{-t(\text{Re}(\sqrt{\lambda})-\nu)}}{\text{Re}(\sqrt{\lambda})-\nu} + \frac{1}{\text{Re}(\sqrt{\lambda})+\nu}\right)\frac{\|R\|_{\mathcal{E}}}{2\sqrt{|\lambda|}} \\ \\

& \leqslant & \dis \frac{1}{p^{1/p} (\text{Re}(\sqrt{\lambda})-\nu)^{1/p}} \frac{1}{q^{1/q} (\text{Re}(\sqrt{\lambda})+\nu)^{1/q}} \, \frac{\|R\|_{\mathcal{E}}}{2\sqrt{|\lambda|}} \\ \ecart
&& \dis + \frac{2}{\text{Re}(\sqrt{\lambda})-\nu}\,\frac{\|R\|_{\mathcal{E}}}{2\sqrt{|\lambda|}} \\ \\

& \leqslant & \dis \frac{2}{\text{Re}(\sqrt{\lambda})-\nu}\,\frac{\|R\|_{\mathcal{E}}}{\sqrt{|\lambda|}}.
\end{array}$$
Let $\lambda = |\lambda|e^{i\arg(\lambda)} \in \Sigma_{\L_1}$, then $|\arg(\lambda)| \leqslant \pi - 2\varepsilon_{\L_1}$. Thus
$$\begin{array}{rcl}
\dis \text{Re}(\sqrt{\lambda})-\nu & \geqslant & \dis \sqrt{|\lambda|} \cos\left(\frac{\arg(\lambda)}{2}\right) - \nu \geqslant \dis \sqrt{|\lambda|} \cos\left(\frac{\pi}{2} - \varepsilon_{\L_1}\right) - \nu \\ \ecart

& \geqslant & \dis \sqrt{|\lambda|} \sin\left(\varepsilon_{\L_1}\right) - \frac{\sqrt{|\lambda|}}{2} \sin\left(\varepsilon_{\L_1}\right)  \geqslant  \dis \frac{\sqrt{|\lambda|}}{2} \sin\left(\varepsilon_{\L_1}\right).
\end{array}$$
Therefore, setting $\dis M_{\L_1} = \frac{4}{\sin\left(\varepsilon_{\L_1}\right)} > 0$ such that
\begin{equation}\label{estim V}
\|V\|_{\mathcal{E}} \leqslant \frac{M_{\L_1}}{|\lambda|}\, \|R\|_{\mathcal{E}}.
\end{equation}
\end{proof}

\subsection{Study of operator $\L_2$}

In this section, in view to determine the spectral properties of $\L_2$, since $\L_2 = -\A$, we will study the spectral properties of $\A$.

We focus ourselves, for $\lambda \leqslant 0$, on the following spectral equation
\begin{equation}\label{eq spectrale}
\A \Psi - \lambda \Psi = F,
\end{equation}
which writes
$$
\left( 
\begin{array}{cc}
0 & 1 \\ 
-\left( \dfrac{\partial ^{2}}{\partial \theta ^{2}}+1\right) ^{2} & -2\left( 
\dfrac{\partial ^{2}}{\partial \theta ^{2}}-1\right)%
\end{array}\right) \left( 
\begin{array}{c}
\psi _{1} \\ 
\psi _{2}%
\end{array}%
\right) - \lambda \left( 
\begin{array}{c}
\psi _{1} \\ 
\psi _{2}%
\end{array}%
\right) = \dbinom{F_{1}}{F_{2}},
$$
with $F_{1}\in W_{0}^{2,p}(0,\omega)$ and $F_{2} \in L^{p}(0,\omega)$. 

We have to find the unique couple $(\psi_1,\psi_2) \in \left(W^{4,p}(0,\omega)\cap W_{0}^{2,p}(0,\omega)\right) \times W_{0}^{2,p}(0,\omega)$, which satisfies the following system
\begin{equation*}
\left\{ 
\begin{array}{lll}
\psi _{2} - \lambda \psi _{1}&=&F_{1} \\ 
- \dfrac{\partial^{4}}{\partial \theta^{4}} \psi_{1} - 2 \dfrac{\partial^2}{\partial \theta^2} \psi_1 - \psi_1 - 2 \dfrac{\partial^{2}}{\partial \theta^{2}} \psi_{2} + 2 \psi_2 - \lambda \psi_{2} &=& F_{2}.
\end{array}
\right.
\end{equation*}
Thus, we first have to solve
\begin{equation*}
\left\{ 
\begin{array}{l}
- \dfrac{\partial^{4}}{\partial \theta^{4}} \psi_{1} - 2 \dfrac{\partial^2}{\partial \theta^2} \psi_1 - \psi_1 - 2 \dfrac{\partial^{2}}{\partial \theta^{2}} (\lambda\psi_{1}+F_1) + (2 - \lambda)(\lambda \psi_{1} + F_1) = F_{2} \\ \ecart 
\psi _{1}\in W^{4,p}(0,\omega)\cap W_{0}^{2,p}(0,\omega),
\end{array}
\right.
\end{equation*}%
that is 
\begin{equation*}
\left\{ 
\begin{array}{l}
- \psi^{(4)}_{1} - 2 (1+\lambda) \psi''_{1} - (\lambda - 1)^2 \psi _{1} = F_{2} + 2 F''_{1} + (\lambda-2) F_{1} \\ \ecart
\psi_1 (0) = \psi_1 (\omega) = \psi'_1 (0) = \psi'_1 (\omega) = 0.
\end{array}
\right.
\end{equation*}
Set $G_\lambda = - F_{2} - 2 (F''_{1}-F_1) - \lambda F_1$, it follows that the previous system writes
\begin{equation}\label{Pb psi}
\left\{ 
\begin{array}{l}
\psi _{1}^{(4)} + 2(\lambda + 1) \psi''_{1} + (\lambda - 1)^{2}\psi_{1} = G_\lambda \\ \ecart
\psi_{1}(0) = \psi_{1}(\omega) = \psi'_{1}(0) = \psi'_{1}(\omega) = 0.
\end{array}
\right.
\end{equation}
Then, the characteristic equation 
\begin{equation*}
\chi^{4} + 2\left(1+\lambda \right) \chi^{2}+\left( \lambda -1\right)^{2}=0,
\end{equation*}
admits, for $\lambda < 0$, the following four distinct solutions
\begin{equation}\label{alpha i}
\left\{ 
\begin{array}{c}
\dis \alpha _{1} = \sqrt{-\lambda} + i, \quad \alpha_3 = - \alpha_1 \\ \ecart
\dis \alpha _{2} = \sqrt{-\lambda} - i, \quad \alpha_4 = - \alpha_2,
\end{array}\right.
\end{equation}
and for $\lambda = 0$, two double solutions that are $i$ and $-i$.  

We have to distinguish the two cases : $\lambda=0$ and $\lambda<0$.

\subsubsection{Case $\lambda = 0$ : Invertibility of $\A$}
 
\begin{Prop}\label{Prop A inversible} $\A$ is invertible with bounded inverse $i.e.$ $0\in \rho(\A)$ and there exists $\varepsilon_0 > 0$, such that $\overline{B(0,\varepsilon_0)} \subset \rho(\A)$.
\end{Prop}
\begin{proof}
Here $\lambda = 0$. We have to solve \eqref{eq spectrale}. This is equivalent to solve \eqref{Pb psi} which writes as
\begin{equation}\label{Pb psi1 A inversible}
\left\{ \begin{array}{l}
\psi _{1}^{(4)}+2\psi''_{1}+\psi _{1}=G_{0} \\ \ecart
\psi _{1}(0)=\psi _{1}(\omega )=\psi'_{1}(0)=\psi'_{1}(\omega )=0.
\end{array}
\right.
\end{equation}
From \cite{thorel}, Theorem 2.8, statement 2., there exists a unique classical solution of problem \eqref{Pb psi1 A inversible}. Thus, using the closed graph theorem, there exists $C_1>0$ such that
$$\|\psi_1''\|_{L^p(0,\omega)} \leqslant C_1 \|G_1\|_{L^p(0,\omega)} \leqslant C_1 \left(\|F_2\|_{L^p(0,\omega)} + 2 \|F_1\|_{W^{2,p}(0,\omega)}\right) \leqslant 2C_1 \|F\|_X,$$
and from the Poincar\'e inequality, there exists $C_\omega > 0$ such that
$$\|\psi_1\|_{W_0^{2,p}(0,\omega)} \leqslant C_\omega \|\psi_1''\|_{L^p(0,\omega)} \leqslant 2 C_1 C_\omega  \|F\|_X.$$
Finally, since $\psi_2 = F_1$, then we have
$$\|\psi_2\|_{L^p(0,\omega)} = \|F_1\|_{L^p(0,\omega)} \leqslant \|F_2\|_{L^p(0,\omega)} + \|F_1\|_{W^{2,p}(0,\omega)} = \|F\|_X.$$
Therefore, there exists a unique solution $\Psi$ of $\A \Psi = F$ with $\|\Psi\|_X \leqslant (1+2C_1C_\omega)\|F\|_X$, from which we deduce that there exists $\varepsilon_0 > 0$, such that $\overline{B(0,\varepsilon_0)} \subset \rho(\A)$. 
\end{proof}

\subsubsection{Case $\lambda < 0$ : Spectral study of $\A$}

In order to prove \refP{Prop estim resolvante A}, we first have to state the following technical results. 
\begin{Lem}\label{Lem I J}
Let $\alpha \in \CC\setminus\{0\}$, $a, b \in \RR$ with $a<b$ and $f \in W^{2,p}_0(a,b)$. For all $x \in [a,b]$, we set 
$$K(x) = \int_a^x e^{-(x-s)\alpha} f(s)~ds + \int_x^b e^{-(s-x)\alpha} f(s)~ds.$$
Then, we have
$$K(x) = \frac{2}{\alpha}\,f(x) + \frac{1}{\alpha^2} \int_a^x e^{-(x-s)\alpha} f''(s)~ds + \frac{1}{\alpha^2}\int_x^b e^{-(s-x)\alpha} f''(s)~ds.$$
\end{Lem}
\begin{proof}
The result is easily obtained by two integrations by parts. 
\end{proof}

\begin{Prop}\label{Prop sol psi 1}
For any $\lambda < 0$ problem \eqref{Pb psi} has a unique solution given by
\begin{equation}\label{psi 1}
\begin{array}{lll}
\psi_1 (\theta) & := & \dis e^{-\theta \alpha_2} (\beta_1 + \beta_2 + \beta_3 + \beta_4) + e^{-(\omega-\theta)\alpha_2} (\beta_3 + \beta_4 - \beta_1 - \beta_2) + S(\theta) \\ \ecart
&& \dis + \left(e^{-\theta \alpha_1} - e^{-\theta \alpha_2}\right) (\beta_2 + \beta_4) + \left(e^{-(\omega-\theta) \alpha_1} - e^{-(\omega-\theta) \alpha_2}\right) (\beta_4 - \beta_2),
\end{array}
\end{equation}
where 
\begin{equation}\label{beta i}
\left\{\begin{array}{rcl}
\beta_1 & = & \dis \frac{1}{4i} U_-^{-1} \, \frac{1-e^{-\omega \alpha_1}}{1 - e^{-\omega \alpha_2}} \left(J(0) - J(\omega)\right) \\ \ecart

\beta_2 & = & \dis - \frac{1}{4i} U_-^{-1} \, \left(J(0) - J(\omega)\right) \\ \ecart

\beta_3 & = & \dis -\frac{1}{4i} U_+^{-1} \, \frac{1+e^{-\omega \alpha_1}}{1 + e^{-\omega \alpha_2}} \left(J(0) + J(\omega)\right) \\ \ecart

\beta_4 & = & \dis \frac{1}{4i} U_+^{-1} \, \left(J(0) + J(\omega)\right),
\end{array}\right.
\end{equation}
with 
\begin{equation}\label{U-V M}
\left \{\begin{array}{rcl}
U_- & := & \dis 1 - e^{-2\omega \sqrt{-\lambda}} - 2\omega \sqrt{-\lambda}\, e^{-\omega \sqrt{-\lambda}} \\ \ecart
U_+ & := & \dis 1 - e^{-2\omega \sqrt{-\lambda}} + 2\omega \sqrt{-\lambda}\, e^{-\omega \sqrt{-\lambda}},
\end{array}\right.
\end{equation}
and $S$ is a particular solution of the equation of problem \eqref{Pb psi} which is given, for all $\theta \in [0,\omega]$, by 
\begin{equation}\label{Sp}
\begin{array}{rcl}
S(\theta) &: = & \dis \frac{e^{-\theta \alpha_2}}{2\alpha_2 \,(1 - e^{-2\omega\alpha_2})}  \left(J(0) - e^{- \omega \alpha_2} J(\omega)\right) - \frac{\lambda}{\alpha_1^2\alpha_2^2}\, F_1(\theta) \\ \ecart 
&& \dis + \frac{e^{-(\omega-\theta) \alpha_2}}{2\alpha_2 \,(1 - e^{-2\omega\alpha_2})} \left( J(\omega) - e^{- \omega \alpha_2} J(0)\right) - \frac{1}{2\alpha_2} \,J(\theta),
\end{array}
\end{equation}
with
\begin{equation}\label{J}
J(\theta) := \int_0^\theta e^{-(\theta-s) \alpha_2} v(s) ~ds + \int_\theta^\omega e^{-(s-\theta)\alpha_2} v(s) ~ds,
\end{equation}
where
\begin{equation}\label{v}
\begin{array}{rll}
v(\theta) & := & \dis \frac{e^{-\theta \alpha_1}}{2\alpha_1\left(1 - e^{- 2\omega \alpha_1} \right)} \left(I(0) - e^{-\omega\alpha_1}I(\omega)\right) + \frac{\lambda}{\alpha_1^2\alpha_2^2}\, F_1''(\theta) \\ \ecart
&& \dis + \frac{e^{-(\omega-\theta) \alpha_1}}{2\alpha_1\left(1 - e^{- 2\omega \alpha_1} \right)} \left(I(\omega) - e^{-\omega\alpha_1}I(0)\right)  - \frac{1}{2\alpha_1}\, I(\theta),
\end{array}
\end{equation}
and
\begin{equation}\label{I}
\begin{array}{lll}
I(\theta) & = & \dis \int_0^\theta e^{-(\theta-s)\alpha_1} \left(-F_2-2(F_1''-F_1)+ \frac{\lambda}{\alpha_1^2} \,F_1''\right)(s)~ ds \\ \ecart
&& \dis + \int_\theta^\omega e^{-(s-\theta)\alpha_1} \left(-F_2 - 2(F_1''-F_1) +\frac{\lambda}{\alpha_1^2} \,F_1''\right)(s)~ ds.
\end{array}
\end{equation}
\end{Prop}

\begin{proof}
In order to apply results obtained in \cite{LLMT} and \cite{LMMT}, we set $L_- = -\alpha_1 I$, $M= -\alpha_2 I$, $r_- = \alpha_1^2 - \alpha_2^2$, $a=0$ and $b=\omega$. Then, problem \eqref{Pb psi} reads as
\begin{equation*}
\left\{ 
\begin{array}{l}
\psi _{1}^{(4)} - (L_-^2 + M^2) \psi''_{1} + L_-^2 M^2 \psi_{1} = G_{\lambda} \\ \ecart
\psi _{1}(0)=\psi _{1}(\omega )=\psi'_{1}(0)=\psi'_{1}(\omega )=0.
\end{array}
\right.
\end{equation*}
From \cite{LMMT}, there exists a unique solution whose representation formula is explicitly given in \cite{LLMT} by (14)-(15)-(16). This representation formula shows that $\psi_1$ writes as in \eqref{psi 1}, with $\beta_1,\beta_2,\beta_3,\beta_4$ given by \eqref{beta i} and
\begin{equation}\label{Sp M}
\begin{array}{rll}
S(\theta) &= &\dis \frac{1}{2\alpha_2} e^{-\theta \alpha_2} Z \int_0^\omega e^{-s \alpha_2} v_0(s)~ ds + \frac{1}{2\alpha_2} e^{-(\omega-\theta)\alpha_2} Z  \int_0^\omega e^{-(\omega-s)\alpha_2} v_0(s)~ ds \\ \ecart
&& - \dis \frac{1}{2\alpha_2} \int_0^\theta e^{-(\theta-s)\alpha_2} v_0(s)~ ds - \frac{1}{2\alpha_2} \int_\theta^\omega e^{-(s-\theta)\alpha_2} v_0(s)~ ds \\ \ecart
&& - \dis\frac{1}{2\alpha_2} e^{-\theta\alpha_2} Z e^{-\omega\alpha_2} \int_0^\omega e^{-(\omega-s)\alpha_2} v_0(s)~ ds \\ \ecart
&&- \dis\frac{1}{2\alpha_2} e^{-(\omega-\theta)\alpha_2} Z e^{-\omega \alpha_2} \int_0^\omega e^{-s\alpha_2} v_0(s)~ ds, \quad \theta \in [0,\omega], 
\end{array}
\end{equation}
where
\begin{equation}\label{v0 M}
\begin{array}{rll}
v_0(\theta) &:= &  \dis\frac{1}{2\alpha_1} e^{-\theta\alpha_1} W \int_0^\omega e^{-s\alpha_1} G_\lambda(s)~ ds + \frac{1}{2\alpha_1} e^{-(\omega-\theta)\alpha_1} W  \int_0^\omega e^{-(\omega-s)\alpha_1} G_\lambda(s)~ ds\\ \ecart
&& - \dis \frac{1}{2\alpha_1} e^{-\theta\alpha_1} W e^{-\omega \alpha_1} \int_0^\omega e^{-(\omega-s)\alpha_1} G_\lambda(s)~ ds \\ \ecart
&& - \dis\frac{1}{2\alpha_1} e^{-(\omega-\theta)\alpha_1} W e^{-\omega\alpha_1} \int_0^\omega e^{-s \alpha_1} G_\lambda(s)~ ds \\ \ecart 
&& - \dis \frac{1}{2\alpha_1} I_1(\theta), \quad \theta \in [0,\omega],
\end{array}
\end{equation}
with $Z := \left(1-e^{-2\omega\alpha_2} \right)^{-1}$, $W := \left(1-e^{-2\omega\alpha_1}\right)^{-1}$ and
\begin{equation}\label{I1}
I_1(\theta) = \int_0^\theta e^{-(\theta-s)\alpha_1} G_\lambda(s)~ ds + \int_\theta^\omega e^{-(s-\theta)\alpha_1} G_\lambda(s)~ ds.
\end{equation}
Then, since $G_\lambda = - F_{2} - 2 (F''_{1}-F_1) - \lambda F_1$, we have 
\begin{equation*}
\begin{array}{rcl}
I_1(\theta) & = & \dis \int_0^\theta e^{-(\theta-s)\alpha_1} \left(-F_2 - 2 (F''_1-F_1)\right)(s) ~ds \\ \ecart
&& \dis + \int_\theta^\omega e^{-(s-\theta)\alpha_1} \left(-F_2 - 2 (F''_1-F_1)\right)(s) ~ ds \\ \ecart
&& \dis - \lambda \int_0^\theta e^{-(\theta-s)\alpha_1} F_1(s)~ ds - \lambda
\int_\theta^\omega e^{-(s-\theta)\alpha_1} F_1(s)~ ds. 
\end{array}
\end{equation*}
Then, from \refL{Lem I J}, since $F_1 \in W^{2,p}_0(0,\omega)$, it follows
\begin{equation*}
\begin{array}{rcl}
I_1(\theta) & = & \dis \int_0^\theta e^{-(\theta-s)\alpha_1} \left(-F_2 - 2 (F''_1 - F_1)\right)(s)~ ds \\ \ecart
&& \dis + \int_\theta^\omega e^{-(s-\theta)\alpha_1} \left(-F_2 - 2 (F''_1 - F_1)\right)(s)~ ds \\ \ecart
&&\dis - \frac{2\lambda}{\alpha_1} F_1(\theta) + \frac{\lambda}{\alpha_1^2} \left( \int_0^\theta e^{-(\theta-s)\alpha_1} F_1''(s)~ ds + \int_\theta^\omega e^{-(s-\theta)\alpha_1} F_1''(s)~ ds \right) \\ \\

& = & \dis - \frac{2\lambda}{\alpha_1} F_1(\theta) + \int_0^\theta e^{-(\theta-s)\alpha_1} \left(-F_2-2(F_1''-F_1)+ \frac{\lambda}{\alpha_1^2} \,F_1''\right)(s)~ ds \\ \ecart
&& \dis + \int_\theta^\omega e^{-(s-\theta)\alpha_1} \left(-F_2 - 2(F_1''-F_1) +\frac{\lambda}{\alpha_1^2} \,F_1''\right)(s)~ ds.
\end{array}
\end{equation*}
Hence $I$, given by \eqref{I}, satisfies
\begin{equation}\label{I1 I}
I(\theta) =  I_1(\theta) + \frac{2\lambda}{\alpha_1}\, F_1(\theta).
\end{equation}
Note that, from \eqref{I1} and \eqref{I1 I}, we have
$$\int_0^\omega e^{-s \alpha_1} G_\lambda(s)~ ds = I_1(0) = I(0) \quad \text{and} \quad \int_0^\omega e^{-(\omega-s) \alpha_1} G_\lambda(s)~ ds = I_1(\omega) = I(\omega).$$
Therefore, from \eqref{v0 M}, for all $\theta \in [0,\omega]$, we deduce that
\begin{equation*}
\begin{array}{rll}
v_0(\theta) & = & \dis \frac{1}{2\alpha_1} \, e^{-\theta \alpha_1} W \left(I(0) - e^{-\omega\alpha_1}I(\omega)\right)  \\ \ecart
&& \dis + \frac{1}{2\alpha_1} \, e^{-(\omega-\theta) \alpha_1} W \left(I(\omega) - e^{-\omega\alpha_1}I(0)\right) - \frac{1}{2\alpha_1}\, I_1(\theta) \\ \\

& = & \dis \frac{1}{2\alpha_1} \, e^{-\theta \alpha_1} W \left(I(0) - e^{-\omega\alpha_1}I(\omega)\right) + \frac{\lambda}{\alpha_1^2} \,F_1(\theta) \\ \ecart
&& \dis + \frac{1}{2\alpha_1} \, e^{-(\omega-\theta) \alpha_1} W \left(I(\omega) - e^{-\omega\alpha_1}I(0)\right)  - \frac{1}{2\alpha_1}\, I(\theta).
\end{array}
\end{equation*}
Moreover, setting 
\begin{equation}\label{v1}
v_1 = v_0 - \frac{\lambda}{\alpha_1^2} \,F_1.
\end{equation}
Thus, for all $\theta \in [0,\omega]$, noting
\begin{equation}\label{J1}
J_1(\theta) = \int_0^\theta e^{-(\theta-s) \alpha_2} v_0(s)~ ds + \int_\theta^\omega e^{-(s-\theta)\alpha_2 } v_0(s)~ ds,
\end{equation}
from \eqref{v1} and \refL{Lem I J}, since $F_1 \in W^{2,p}_0(0,\omega)$, we obtain
\begin{equation*}
\begin{array}{rll}
J_1(\theta) & = & \dis \int_0^\theta e^{-(\theta-s) \alpha_2} v_1(s)~ ds + \int_\theta^\omega e^{-(s-\theta)\alpha_2 } v_1(s)~ ds \\ \ecart
&& \dis + \frac{\lambda}{\alpha_1^2} \left(\int_0^\theta e^{-(\theta-s) \alpha_2} F_1(s)~ ds + \int_\theta^\omega e^{-(s-\theta)\alpha_2 } F_1(s)~ ds \right),
\end{array}
\end{equation*}
hence
\begin{equation*}
\begin{array}{rll}
J_1(\theta) & = & \dis \int_0^\theta e^{-(\theta-s) \alpha_2} v_1(s)~ ds + \int_\theta^\omega e^{-(s-\theta)\alpha_2 } v_1(s)~ ds + \frac{2\lambda}{\alpha_1^2\alpha_2} F_1(\theta) \\ \ecart
&& \dis + \frac{\lambda }{\alpha_1^2\alpha_2^2}\left( \int_0^\theta e^{-(\theta-s) \alpha_2} F_1''(s)~ ds + \int_\theta^\omega e^{-(s-\theta)\alpha_2 } F_1''(s)~ ds \right) \\ \\

& = & \dis \int_0^\theta e^{-(\theta-s) \alpha_2} \left(v_1(s) + \frac{\lambda}{\alpha_1^2\alpha_2^2} \,F_1''(s)\right) ds \\ \ecart
&& \dis + \int_\theta^\omega e^{-(s-\theta)\alpha_2 } \left(v_1(s) + \frac{\lambda}{\alpha_1^2\alpha_2^2}\,F_1''(s)\right) ds + \frac{2\lambda}{\alpha_1^2\alpha_2} F_1(\theta).
\end{array}
\end{equation*}
From \eqref{v1}, for all $\theta \in [0,\omega]$, we deduce that $v$ given by \eqref{v} and $J$ given by \eqref{J}, satisfy 
$$v(\theta) = v_1(\theta) + \frac{\lambda}{\alpha_1^2\alpha_2^2} \, F_1''(\theta),$$
and 
\begin{equation}\label{J1 J}
J(\theta) =  J_1(\theta) - \frac{2\lambda}{\alpha_1^2\alpha_2} F_1(\theta).
\end{equation}
Note that, from \eqref{J1} and \eqref{J1 J}, one has
$$\int_0^\omega e^{-s\alpha_2 } v_0(s)~ ds = J_1(0) = J(0) \quad \text{and} \quad \int_0^\omega e^{-(\omega-s)\alpha_2 } v_0(s)~ ds = J_1(\omega) = J(\omega).$$
Finally, from \eqref{J}, \eqref{Sp M} and \eqref{J1 J}, for all $\theta \in [0,\omega]$, we deduce that
\begin{equation*}
\begin{array}{rll}
S(\theta) & = & \dis \frac{1}{2\alpha_2}\, e^{-\theta \alpha_2} Z \left(J(0) - e^{- \omega \alpha_2} J(\omega)\right) \\ \ecart 
&& \dis + \frac{1}{2\alpha_2} \,e^{-(\omega-\theta) \alpha_2} Z \left( J(\omega) - e^{- \omega \alpha_2} J(0)\right) - \frac{1}{2\alpha_2} \, J_1(\theta),
\end{array}
\end{equation*}
which leads to \eqref{Sp}. 

The constants $\beta_i$, $i=1,2,3,4$, are given by (15) and (16) in \cite{LLMT}. $U_-$, $V_-$ and $F'_-(a) \pm F'_-(\gamma)$ in \cite{LLMT} are replaced here by $U_-$, $U_+$ and $S'(0) \pm S'(\omega)$. So, in order to compute constants $\beta_i$, we now make explicit $U_-$, $U_+$ and $S'(0) \pm S'(\omega)$.
\begin{equation*}
\begin{array}{rll}
U_- & = & \dis 1 - e^{-\omega (\alpha_1 + \alpha_2)} - (\alpha_1^2 - \alpha_2^2)^{-1} (\alpha_1 + \alpha_2)^2 \left(e^{-\omega \alpha_2} - e^{-\omega \alpha_1}\right) \\ \ecart
& = & \dis 1 - e^{-2\omega \sqrt{-\lambda}} + i\sqrt{-\lambda} \left(e^{-\omega (\sqrt{-\lambda} -i)} - e^{-\omega (\sqrt{-\lambda} + i)}\right) \\ \ecart

& = & \dis 1 - e^{-2\omega \sqrt{-\lambda}}  + 2i \sqrt{-\lambda}\, e^{-\omega \sqrt{-\lambda}} \left( \frac{e^{\omega i } - e^{-\omega i}}{2} \right) \\ \ecart

& = & \dis 1 - e^{-2\omega \sqrt{-\lambda}} - 2 \sqrt{-\lambda} \,e^{-\omega \sqrt{-\lambda}} \sin(\omega),
\end{array}
\end{equation*}
and
\begin{equation*}
\begin{array}{rll}
U_+ & = & \dis 1 - e^{-\omega (\alpha_1 + \alpha_2)} + (\alpha_1^2 - \alpha_2^2)^{-1} (\alpha_1 + \alpha_2)^2 \left(e^{-\omega \alpha_2} - e^{-\omega \alpha_1}\right) \\ \ecart
& = & \dis 1 - e^{-2\omega \sqrt{-\lambda}} - i \sqrt{-\lambda} \left(e^{-\omega (\sqrt{-\lambda} -i)} - e^{-\omega (\sqrt{-\lambda} + i)}\right) \\ \ecart

& = & \dis 1 - e^{-2\omega \sqrt{-\lambda}} + 2 \sqrt{-\lambda}\, e^{-\omega \sqrt{-\lambda}} \sin(\omega).
\end{array}
\end{equation*}
From \eqref{Sp}, it follows that
\begin{equation*}
\begin{array}{rll}
S'(\theta) & = &  \dis -\frac{1}{2} \,e^{-\theta \alpha_2} Z \left(J(0) - e^{- \omega \alpha_2} J(\omega)\right) - \frac{\lambda}{\alpha_1^2\alpha_2^2}\, F_1'(\theta) \\ \ecart 
&& \dis + \frac{1}{2}\, e^{-(\omega-\theta) \alpha_2} Z \left( J(\omega) - e^{- \omega \alpha_2} J(0)\right) - \frac{1}{2\alpha_2} \, J'(\theta) \\ \\

& = & \dis -\frac{1}{2} \, e^{-\theta \alpha_2} Z \left(J(0) - e^{- \omega \alpha_2} J(\omega)\right) + \frac{1}{2} \, e^{-(\omega-\theta) \alpha_2} Z \left( J(\omega) - e^{- \omega \alpha_2} J(0)\right) \\ \ecart 
&& \dis + \frac{1}{2} \left(\int_0^\theta e^{-(\theta-s) \alpha_2} v(s) ~ds - \int_\theta^\omega e^{-(s-\theta)\alpha_2 } v(s) ~ds\right) - \frac{\lambda}{\alpha_1^2\alpha_2^2}\, F_1'(\theta),
\end{array}
\end{equation*}
from which we deduce that
$$S'(0) + S'(\omega) = - \frac{J(0) - J(\omega)}{\left(1 - e^{-\omega \alpha_2}\right)} \quad \text{and} \quad S'(0) - S'(\omega) = - \frac{J(0) + J(\omega)}{\left(1 + e^{-\omega \alpha_2}\right)}.$$
This prove that constants $\beta_i$, $i=1,2,3,4$, are given by \eqref{beta i}.
\end{proof}

\begin{Rem}
Since $0 < \sin(\omega) < \omega$, for all $\omega > 0$, then we have
\begin{equation*}
U_- = 1 - e^{-2\omega \sqrt{-\lambda}} - 2 \sqrt{-\lambda} \,e^{-\omega \sqrt{-\lambda}} \sin(\omega) \geqslant 1 - e^{-2\omega \sqrt{-\lambda}} - 2 \omega\sqrt{-\lambda} \, e^{-\omega \sqrt{-\lambda}},
\end{equation*}
and
\begin{equation*}
U_+ = 1 - e^{-2\omega \sqrt{-\lambda}} + 2 \sqrt{-\lambda}\, e^{-\omega \sqrt{-\lambda}} \sin(\omega) \geqslant 1 - e^{-2\omega \sqrt{-\lambda}} - 2 \omega\sqrt{-\lambda}\, e^{-\omega \sqrt{-\lambda}}.
\end{equation*}
Let $x>0$. Setting
$$f(x) = 1 - e^{-2x} - 2 x e^{-x},$$
we have
$$f'(x) = 2 e^{-2x} - 2 e^{-x} + 2 x e^{-x} = 2 e^{-x} \left(e^{-x} + x - 1\right) > 0.$$
It follows that $f(x) > f(0) = 0$. Finally, for all $\omega,\sqrt{-\lambda} > 0$, we deduce that 
\begin{equation}\label{U > 0}
U_- \geqslant 1 - e^{-2\omega \sqrt{-\lambda}} - 2 \omega\sqrt{-\lambda} e^{-\omega \sqrt{-\lambda}} = f(\omega\sqrt{-\lambda}) > 0,
\end{equation}
and 
\begin{equation}\label{V > 0} 
U_+ \geqslant 1 - e^{-2\omega \sqrt{-\lambda}} - 2 \omega\sqrt{-\lambda} e^{-\omega \sqrt{-\lambda}} = f(\omega\sqrt{-\lambda}) > 0.
\end{equation}
\end{Rem}

\begin{Lem}\label{Lem estim J, I, v}
Let $F_1 \in W^{2,p}_0(0,\omega)$ and $F_2 \in L^p(0,\omega)$. Consider $\varepsilon_0$ defined in \refP{Prop A inversible}, then for all $\lambda \leqslant -\varepsilon_0$, $J$, $v$ and $I$, given by \eqref{J}, \eqref{v} and \eqref{I}, satisfy the following estimates
\begin{enumerate}

\item $\dis\|I\|_{L^p(0,\omega)} \leqslant \frac{2}{\sqrt{-\lambda}} \left(\|F_2\|_{L^p(0,\omega)} + 2\|F_1\|_{L^p(0,\omega)} + 3 \|F_1''\|_{L^p(0,\omega)}\right)$.

\item $\dis |I(0)| + |I(\omega)| \leqslant \frac{2}{\sqrt{-\lambda}^{1-1/p}} \, \left(\|F_2\|_{L^p(0,\omega)} + 2\|F_1\|_{L^p(0,\omega)} + 3 \|F_1''\|_{L^p(0,\omega)}\right).$

\item $\dis \|v\|_{L^p(0,\omega)} \leqslant \frac{M_1}{-\lambda} \left(\|F_2\|_{L^p(0,\omega)} + 2 \|F_1\|_{L^p(0,\omega)} + 3 \|F_1''\|_{L^p(0,\omega)}\right)$,

where $M_1 = 2 + \frac{2}{1 - e^{- 2\omega \sqrt{\varepsilon_0}}}$.

\item $\dis \|J \|_{L^p(0,\omega)} \leqslant \frac{2}{\sqrt{-\lambda}}\, \|v\|_{L^p(0,\omega)}$.

\item $\dis |J(0)| + |J(\omega)| \leqslant \frac{2}{\sqrt{-\lambda}^{1-1/p}}\, \|v\|_{L^p(0,\omega)}$.
\end{enumerate}
\end{Lem}
\begin{proof}\hfill
\begin{enumerate}
\item From \eqref{alpha i} and \eqref{I}, we obtain
\begin{equation*}
\begin{array}{rcl}
\|I\|_{L^p(0,\omega)} & \leqslant & \dis \sup_{\theta \in [0,\omega]} \left(\int_0^\theta e^{-(\theta-s)\sqrt{-\lambda}}~ ds + \int_\theta^\omega e^{-(s-\theta)\sqrt{-\lambda}} ~ds \right) \|2F_1 - F_2\|_{L^p(0,\omega)}   \\ \ecart
&&\dis + \sup_{\theta \in [0,\omega]} \left(\int_0^\theta e^{-(\theta-s)\sqrt{-\lambda}}~ ds + \int_\theta^\omega e^{-(s-\theta)\sqrt{-\lambda}} ~ds \right) \left\|\left(\frac{\lambda}{\alpha_1^2}-2\right) F_1''\right\|_{L^p(0,\omega)} \\ \\

& \leqslant & \dis \sup_{\theta \in [0,\omega]}\left( \frac{1 - e^{-\theta \sqrt{-\lambda}}}{\sqrt{-\lambda}} + \frac{1 - e^{-(\omega-\theta)\sqrt{-\lambda} }}{\sqrt{-\lambda}}\right) \left(\|F_2\|_{L^p(0,\omega)} + 2\|F_1\|_{L^p(0,\omega)}\right) \\ \ecart
&& \dis +\, 3 \sup_{\theta \in [0,\omega]}\left( \frac{1 - e^{-\theta \sqrt{-\lambda}}}{\sqrt{-\lambda}} + \frac{1 - e^{-(\omega-\theta)\sqrt{-\lambda} }}{\sqrt{-\lambda}}\right) \|F_1''\|_{L^p(0,\omega)} \\ \\

& \leqslant & \dis \frac{2}{\sqrt{-\lambda}} \left(\|F_2\|_{L^p(0,\omega)} + 2\|F_1\|_{L^p(0,\omega)} + 3 \|F_1''\|_{L^p(0,\omega)}\right).
\end{array}
\end{equation*}

\item Due to \eqref{alpha i} and \eqref{I}, from the Hölder inequality, it follows
\begin{equation*}
\begin{array}{rcl}
|I(0)| + |I(\omega)| & \leqslant & \dis \int_0^\omega e^{-s\sqrt{-\lambda}} \left|-F_2(s) - 2(F_1''(s)-F_1(s))+ \frac{\lambda}{\alpha_1^2} \,F_1''(s)\right| ds \\ \ecart
&& \dis + \int_0^\omega e^{-(\omega-s)\sqrt{-\lambda}} \left|-F_2(s) - 2(F_1''(s)-F_1(s))+ \frac{\lambda}{\alpha_1^2} \,F_1''(s)\right| ds \\ \\

& \leqslant & \dis \left(\int_0^\omega e^{-q(\omega-s)\sqrt{-\lambda}} ~ds \right)^{1/q}  \left(\|F_2\|_{L^p(0,\omega)} + 2\|F_1\|_{L^p(0,\omega)} + 3 \|F_1''\|_{L^p(0,\omega)}\right) \\ \ecart
&& \dis + \left(\int_0^\omega e^{-sq\sqrt{-\lambda}}~ds\right)^{1/q}  \left(\|F_2\|_{L^p(0,\omega)} + 2\|F_1\|_{L^p(0,\omega)} + 3 \|F_1''\|_{L^p(0,\omega)}\right)  \\ \\

& \leqslant & \dis \frac{2\left(1 - e^{-\omega q \sqrt{-\lambda}}\right)^{1/q}}{q^{1/q} \sqrt{-\lambda}^{1/q}} \, \left(\|F_2\|_{L^p(0,\omega)} + 2\|F_1\|_{L^p(0,\omega)} + 3 \|F_1''\|_{L^p(0,\omega)}\right) \\ \\

& \leqslant & \dis \frac{2}{\sqrt{-\lambda}^{1-1/p}} \, \left(\|F_2\|_{L^p(0,\omega)} + 2\|F_1\|_{L^p(0,\omega)} + 3 \|F_1''\|_{L^p(0,\omega)}\right).
\end{array}
\end{equation*}

\item From \eqref{alpha i} and \eqref{v}, since $|\alpha_1| = |\alpha_2| = \sqrt{1-\lambda} > \sqrt{-\lambda}$, we have
$$\begin{array}{rcl}
\|v\|_{L^p(0,\omega)} & \leqslant & \dis  \frac{|I(0)| + |I(\omega)|}{2\sqrt{1-\lambda}\left(1 - e^{- 2\omega \sqrt{\varepsilon_0}} \right)} \left( \int_0^\omega e^{-p\theta \sqrt{-\lambda}}~d\theta\right)^{1/p} + \frac{-\lambda}{(1-\lambda)^2} \, \|F_1''\|_{L^p(0,\omega)} \\ \ecart
&& \dis + \frac{|I(0)| + |I(\omega)|}{2\sqrt{1-\lambda}\left(1 - e^{- 2\omega \sqrt{\varepsilon_0}} \right)} \left( \int_0^\omega e^{-p(\omega-\theta) \sqrt{-\lambda}}~d\theta\right)^{1/p} + \frac{1}{2\sqrt{1-\lambda}}\, \|I\|_{L^p(0,\omega)} \\ \\

& \leqslant & \dis \frac{\left(|I(0)| + |I(\omega)|\right)}{\sqrt{1-\lambda}\,\sqrt{-\lambda}^{1/p} p^{1/p} \left(1 - e^{- 2\omega \sqrt{\varepsilon_0}} \right)} + \frac{1}{2\sqrt{1-\lambda}}\, \|I\|_{L^p(0,\omega)} \\ \ecart
&& \dis + \frac{1}{1-\lambda}\|F_1''\|_{L^p(0,\omega)}, 
\end{array}$$
hence
$$\begin{array}{rcl}
\|v\|_{L^p(0,\omega)} & \leqslant & \dis \frac{2 \left(\|F_2\|_{L^p(0,\omega)} + 2\|F_1\|_{L^p(0,\omega)} + 3 \|F_1''\|_{L^p(0,\omega)}\right)}{\sqrt{1-\lambda}\,\sqrt{-\lambda}^{1/p+1/q}\, q^{1/q}\, p^{1/p} \left(1 - e^{- 2\omega \sqrt{\varepsilon_0}} \right)}  \\ \ecart
&& \dis + \frac{\|F_2\|_{L^p(0,\omega)} + 2\|F_1\|_{L^p(0,\omega)} + 3 \|F_1''\|_{L^p(0,\omega)}}{\sqrt{1-\lambda}\,\sqrt{-\lambda}} + \frac{1}{1-\lambda} \, \|F_1''\|_{L^p(0,\omega)} \\ \\

& \leqslant & \dis\frac{M_1}{\sqrt{1-\lambda}\,\sqrt{-\lambda}} \left(\|F_2\|_{L^p(0,\omega)} + 2 \|F_1\|_{L^p(0,\omega)} + 3 \|F_1''\|_{L^p(0,\omega)}\right).
\end{array}$$

\item From \eqref{alpha i} and \eqref{J}, we have
$$\begin{array}{rcl}
\|J \|_{L^p(0,\omega)} & \leqslant & \dis \sup_{\theta \in [0,\omega]}\left( \int_0^\theta e^{-(\theta-s) \sqrt{-\lambda}}~ds + \int_\theta^\omega e^{-(s-\theta)\sqrt{-\lambda} }~ds\right) \|v\|_{L^p(0,\omega)} \\ \ecart

& \leqslant & \dis \sup_{\theta \in [0,\omega]}\left( \frac{1 - e^{-\theta \sqrt{-\lambda}}}{\sqrt{-\lambda}} + \frac{1 - e^{-(\omega-\theta)\sqrt{-\lambda} }}{\sqrt{-\lambda}}\right) \|v\|_{L^p(0,\omega)} \\ \ecart

& \leqslant & \dis\frac{2}{\sqrt{-\lambda}}\, \|v\|_{L^p(0,\omega)}.
\end{array}$$

\item Due to \eqref{alpha i} and \eqref{J}, from the Hölder inequality, we deduce that
$$\begin{array}{rcl}
|J(0)| + |J(\omega)| & \leqslant & \dis \int_0^\omega e^{-s\sqrt{-\lambda} } |v(s)| ~ds + \int_0^\omega e^{-(\omega-s) \sqrt{-\lambda}} |v(s)| ~ds \\ \ecart

& \leqslant & \dis \left(\left(\int_0^\omega e^{-sq\sqrt{-\lambda} } ~ds\right)^{1/q} + \left(\int_0^\omega e^{-(\omega-s)q\sqrt{-\lambda}} ~ds\right)^{1/q}\right) \|v\|_{L^p(0,\omega)} \\ \ecart

& \leqslant & \dis \frac{2\left(1 - e^{-\omega q \sqrt{-\lambda}}\right)^{1/q}}{q^{1/q} \sqrt{-\lambda}^{1/q}}\, \|v\|_{L^p(0,\omega)} \\ \ecart

& \leqslant & \dis\frac{2}{\sqrt{-\lambda}^{1-1/p}}\, \|v\|_{L^p(0,\omega)}.
\end{array}$$
\end{enumerate}
\end{proof}

\begin{Lem}\label{Lem estim norm Lp difference exp}
Let $\lambda < 0$. Then, we have
$$\left(\int_0^\omega \left|e^{-\theta\alpha_1} - e^{-\theta\alpha_2}\right|^p~d\theta\right)^{1/p} \leqslant \frac{4}{\sqrt{-\lambda}^{1+1/p}},$$
and
$$\left(\int_0^\omega \left|e^{-(\omega-\theta)\alpha_1} - e^{-(\omega-\theta)\alpha_2}\right|^p~d\theta\right)^{1/p} \leqslant \frac{4}{\sqrt{-\lambda}^{1+1/p}}.$$
\end{Lem}

\begin{proof}
For $x\geqslant 0$, we have $e^{-\frac{px}{2}}x^p <1$, so
$$\int_0^{+\infty} e^{-px}x^p~ dx = \int_0^{+\infty} e^{-\frac{px}{2}}e^{-\frac{px}{2}}x^p~dx \leqslant \int_0^{+\infty} e^{-\frac{px}{2}}~dx = \frac{2}{p}.$$
Then, from \eqref{alpha i}, we have
$$\int_0^\omega \left|e^{-\theta\alpha_1} - e^{-\theta\alpha_2}\right|^p~d\theta = \int_0^\omega \left|e^{-\theta\sqrt{-\lambda}} \left(e^{-\theta i} - e^{\theta i}\right)\right|^p~d\theta = 2^p \int_0^\omega e^{-p\theta\sqrt{-\lambda}} |\sin(\theta)|^p~d\theta,$$
hence, setting $x=\theta\sqrt{-\lambda}$, it follows that
$$\begin{array}{rcl}
\dis \int_0^\omega \left|e^{-\theta\alpha_1} - e^{-\theta\alpha_2}\right|^p~d\theta & = & \dis 2^p \int_0^{\omega\sqrt{-\lambda}} e^{-p x} \left|\sin\left(\frac{x}{\sqrt{-\lambda}}\right)\right|^p ~\frac{dx}{\sqrt{-\lambda}} \\ \ecart

& \leqslant & \dis \frac{2^p}{\sqrt{-\lambda}} \int_0^{\omega\sqrt{-\lambda}} e^{-p x} \left(\frac{x}{\sqrt{-\lambda}}\right)^p ~dx \\ \ecart

& \leqslant & \dis \frac{2^p}{\sqrt{-\lambda}^{p+1}} \int_0^{+\infty} e^{-p x} x^p ~dx \\ \ecart

& \leqslant & \dis \frac{2^{p+1}}{p\sqrt{-\lambda}^{p+1}} \leqslant \frac{2^{2p}}{\sqrt{-\lambda}^{p+1}}.
\end{array}$$
The second estimate is obtained by change of variable, taking $\omega-\theta$ instead of $\theta$.  
\end{proof}

\begin{Lem}\label{Lem estimation beta i}
Let $F_1 \in W^{2,p}_0(0,\omega)$ and $F_2 \in L^p(0,\omega)$. Consider $\varepsilon_0$ defined in \refP{Prop A inversible}, then for all $\lambda \leqslant -\varepsilon_0$, $\beta_1$, $\beta_2$, $\beta_3$ and $\beta_4$, defined by \eqref{beta i}, satisfy
$$\max\left(|\beta_1+\beta_2|, |\beta_3+\beta_4| \right) \leqslant \frac{M_1 \left(\|F_2\|_{L^p(0,\omega)} + 2 \|F_1\|_{L^p(0,\omega)} + 3 \|F_1''\|_{L^p(0,\omega)}\right)}{\omega (-\lambda)\sqrt{-\lambda}^{2-1/p}\left(1 - e^{-2\omega \sqrt{\varepsilon_0}} - 2\omega \sqrt{\varepsilon_0}\, e^{-\omega \sqrt{\varepsilon_0}}\right)\left(1-e^{-\omega\sqrt{\varepsilon_0}}\right)},$$
and
$$\max\left(|\beta_2|, |\beta_4| \right) \leqslant \frac{M_1 \left(\|F_2\|_{L^p(0,\omega)} + 2 \|F_1\|_{L^p(0,\omega)} + 3 \|F_1''\|_{L^p(0,\omega)}\right)}{2(-\lambda)\sqrt{-\lambda}^{1-1/p}\left(1 - e^{-2\omega \sqrt{\varepsilon_0}} - 2\omega \sqrt{\varepsilon_0}\, e^{-\omega \sqrt{\varepsilon_0}}\right)},$$
where $M_1 = 2 + \frac{2}{1 - e^{- 2\omega \sqrt{\varepsilon_0}}}$.
\end{Lem}

\begin{proof}
Recall that $\beta_i$, $i=1,2,3,4$, depends on $U_-^{-1}$ and $U_+^{-1}$. From \eqref{U > 0} and \eqref{V > 0}, it follows
$$U_- \geqslant f(\omega\sqrt{-\lambda}) \geqslant f(\omega\sqrt{\varepsilon_0}) = 1 - e^{-2\omega \sqrt{\varepsilon_0}} - 2 \omega\sqrt{\varepsilon_0} \, e^{-\omega \sqrt{\varepsilon_0}} > 0,$$
and 
$$U_+ \geqslant f(\omega\sqrt{-\lambda}) \geqslant f(\omega\sqrt{\varepsilon_0}) = 1 - e^{-2\omega \sqrt{\varepsilon_0}} - 2 \omega\sqrt{\varepsilon_0} \,e^{-\omega \sqrt{\varepsilon_0}} > 0.$$
Thus, we deduce
$$U_-^{-1} \leqslant \frac{1}{1 - e^{-2\omega \sqrt{\varepsilon_0}} - 2 \omega\sqrt{\varepsilon_0}\, e^{-\omega \sqrt{\varepsilon_0}}} \quad \text{and} \quad U_+^{-1} \leqslant \frac{1}{1 - e^{-2\omega \sqrt{\varepsilon_0}} - 2 \omega\sqrt{\varepsilon_0} \,e^{-\omega \sqrt{\varepsilon_0}}}.$$
Therefore, from \eqref{beta i} and \refL{Lem estim J, I, v}, we have
\begin{equation*}
\begin{array}{rcl}
|\beta_1+\beta_2| & \leqslant & \dis \frac{|J(0)|+|J(\omega)|}{4\left(1 - e^{-2\omega \sqrt{\varepsilon_0}} - 2\omega \sqrt{\varepsilon_0}\, e^{-\omega \sqrt{\varepsilon_0}}\right)} \, \left| \frac{1-e^{-\omega\alpha_1}}{1-e^{-\omega\alpha_2}} - 1\right| \\ \ecart

& \leqslant & \dis \frac{\|v\|_{L^p(0,\omega)}}{2\sqrt{-\lambda}^{1-1/p}\left(1 - e^{-2\omega \sqrt{\varepsilon_0}} - 2\omega \sqrt{\varepsilon_0}\, e^{-\omega \sqrt{\varepsilon_0}}\right)} \, \left| \frac{e^{-\omega\alpha_2}-e^{-\omega\alpha_1}}{1-e^{-\omega\alpha_2}}\right| \\ \ecart

& \leqslant & \dis \frac{M_1 \left(\|F_2\|_{L^p(0,\omega)} + 2 \|F_1\|_{L^p(0,\omega)} + 3 \|F_1''\|_{L^p(0,\omega)}\right)}{2 (-\lambda)\sqrt{-\lambda}^{1-1/p}\left(1 - e^{-2\omega \sqrt{\varepsilon_0}} - 2\omega \sqrt{\varepsilon_0}\, e^{-\omega \sqrt{\varepsilon_0}}\right)} \, \frac{2e^{-\omega\sqrt{-\lambda}}}{1-e^{-\omega\sqrt{\varepsilon_0}}} \\ \ecart

& \leqslant & \dis \frac{M_1 \left(\|F_2\|_{L^p(0,\omega)} + 2 \|F_1\|_{L^p(0,\omega)} + 3 \|F_1''\|_{L^p(0,\omega)}\right)}{\omega (-\lambda)\sqrt{-\lambda}^{2-1/p}\left(1 - e^{-2\omega \sqrt{\varepsilon_0}} - 2\omega \sqrt{\varepsilon_0}\, e^{-\omega \sqrt{\varepsilon_0}}\right)\left(1-e^{-\omega\sqrt{\varepsilon_0}}\right)} 
\end{array}
\end{equation*}
and similarly
\begin{equation*}
\begin{array}{rcl}
|\beta_3 + \beta_4| & \leqslant & \dis \frac{|J(0)|+|J(\omega)|}{4\left(1 - e^{-2\omega \sqrt{\varepsilon_0}} - 2\omega \sqrt{\varepsilon_0}\, e^{-\omega \sqrt{\varepsilon_0}}\right)} \, \left|1 - \frac{1+e^{-\omega\alpha_1}}{1+e^{-\omega\alpha_2}}\right| \\ \ecart

& \leqslant & \dis \frac{M_1 \left(\|F_2\|_{L^p(0,\omega)} + 2 \|F_1\|_{L^p(0,\omega)} + 3 \|F_1''\|_{L^p(0,\omega)}\right)}{\omega (-\lambda)\sqrt{-\lambda}^{2-1/p}\left(1 - e^{-2\omega \sqrt{\varepsilon_0}} - 2\omega \sqrt{\varepsilon_0}\, e^{-\omega \sqrt{\varepsilon_0}}\right)\left(1-e^{-\omega\sqrt{\varepsilon_0}}\right)}.
\end{array}
\end{equation*}
In the same way, we obtain 
\begin{equation*}
|\beta_2| \leqslant \frac{M_1 \left(\|F_2\|_{L^p(0,\omega)} + 2 \|F_1\|_{L^p(0,\omega)} + 3 \|F_1''\|_{L^p(0,\omega)}\right)}{2(-\lambda)\sqrt{-\lambda}^{1-1/p}\left(1 - e^{-2\omega \sqrt{\varepsilon_0}} - 2\omega \sqrt{\varepsilon_0}\, e^{-\omega \sqrt{\varepsilon_0}}\right)},
\end{equation*}
and
\begin{equation*}
|\beta_4| \leqslant \frac{M_1 \left(\|F_2\|_{L^p(0,\omega)} + 2 \|F_1\|_{L^p(0,\omega)} + 3 \|F_1''\|_{L^p(0,\omega)}\right)}{2(-\lambda)\sqrt{-\lambda}^{1-1/p}\left(1 - e^{-2\omega \sqrt{\varepsilon_0}} - 2\omega \sqrt{\varepsilon_0}\, e^{-\omega \sqrt{\varepsilon_0}}\right)}.
\end{equation*}

\end{proof}

The following proposition will allow us to use the well-defined operator $-\sqrt{\A}$ which, generates a uniformly bounded analytic semigroup $\left(e^{-s\sqrt{\A}}\right)_{s\geqslant 0}$.

\begin{Prop}\label{Prop estim resolvante A}
$\A$ is closed and densely defined in $X$. Moreover, there exists a constant $M > 0$ such that for all $\lambda \leqslant 0$, operator $\A-\lambda I$ is invertible with bounded inverse and 
\begin{equation*}
\left\Vert \left(\A-\lambda I\right)^{-1}\right\Vert_{\L(X)} \leqslant \frac{M}{1+|\lambda|}.
\end{equation*}
\end{Prop}

\begin{proof}
It is clear that $\mathcal{D}(0,\omega)\times \mathcal{D}(0,\omega)\subset D(\mathcal{A})\subset X=W_{0}^{2,p}(0,\omega)\times L^{p}(0,\omega)$, where $\mathcal{D}(0,\omega)$ is the set of $C^{\infty}$-functions with compact support in $(0,\omega)$. Since $\mathcal{D}(0,\omega)$ is dense in each spaces $W_{0}^{2,p}(0,\omega)$ and $L^{p}(0,\omega)$ for their respective norms, then $D(\mathcal{A})$ is dense.

Let $F_1 \in W_{0}^{2,p}(0,\omega)$ and $F_2 \in L^p(0,\omega)$. From \refP{Prop A inversible}, $0 \in \rho(\A)$. From \refP{Prop sol psi 1}, for all $\lambda < 0$, there exist a unique couple 
$$(\psi_1, \psi_2) \in \left(W^{4,p}(0,\omega)\cap W_{0}^{2,p}(0,\omega)\right) \times W_{0}^{2,p}(0,\omega)$$
which satisfies 
\begin{equation}\label{syst psi_1 psi_2}
\left\{ 
\begin{array}{lll}
\psi_2 &=& \lambda \psi _{1} + F_{1} \\ 
\psi _{1}^{(4)} + 2\lambda \psi''_{1} + \lambda^{2}\psi_{1} &=& G_\lambda, 
\end{array}
\right.
\end{equation}
where $G_\lambda = -F_2 - 2(F''_1-F_1) - \lambda F_1$. Recall that $\Psi = (\A - \lambda I)^{-1}F$ reads as \eqref{syst psi_1 psi_2}, then $\RR_- \subset \rho(\A)$, thus $\A$ is closed. 

Moreover, $\psi_1$ is given by \eqref{psi 1}-\eqref{beta i}-\eqref{Sp} and $\psi_2$ is given by
\begin{equation}\label{psi2}
\begin{array}{lll}
\psi_2 (\theta) & := & \dis \lambda e^{-\theta \alpha_2} (\beta_1 + \beta_2 + \beta_3 + \beta_4) + \lambda e^{-(\omega-\theta)\alpha_2} (\beta_3 + \beta_4 - \beta_1 - \beta_2) \\ \ecart
&& \dis + \lambda \left(e^{-\theta \alpha_1} - e^{-\theta \alpha_2}\right) (\beta_2 + \beta_4) + \lambda \left(e^{-(\omega-\theta) \alpha_1} - e^{-(\omega-\theta) \alpha_2}\right) (\beta_4 - \beta_2) \\ \ecart
&& \dis + \lambda S(\theta) + F_1(\theta),
\end{array}
\end{equation}
where $\beta_i$, $i=1,2,3,4$ are given by \eqref{beta i}-\eqref{U-V M}. From \eqref{Sp}, $\lambda S(\theta) + F_1(\theta)$ is given by
\begin{equation}\label{lambda Sp + F1}
\begin{array}{rcl}
\lambda S(\theta) + F_1(\theta) & = & \dis \frac{\lambda }{2\alpha_2\left(1 - e^{- 2\omega \alpha_2} \right)} \,e^{-\theta \alpha_2} \left(J(0) - e^{- \omega \alpha_2} J(\omega)\right) \\ \ecart 
&& \dis + \frac{\lambda }{2\alpha_2\left(1 - e^{- 2\omega \alpha_2} \right)} \, e^{-(\omega-\theta) \alpha_2}\left( J(\omega) - e^{- \omega \alpha_2} J(0)\right) \\ \ecart
&& \dis - \frac{\lambda^2}{\alpha_1^2\alpha_2^2}\, F_1(\theta) + F_1(\theta) - \frac{\lambda}{2\alpha_2}\, J(\theta),
\end{array}
\end{equation}
with $J(\theta)$ given by \eqref{J}. Our aim is to prove that, for all $\lambda \leqslant 0$, there exists $M > 0$, such that
$$\|(\A - \lambda I)^{-1} F\|_{\L(X)} \leqslant \frac{M}{1+|\lambda|} \|F\|_X,$$
where
\begin{equation}\label{Def norme X}
\|F\|_X = \left\|\left(\begin{array}{c}
F_1 \\
F_2
\end{array}\right)\right\|_X = \|F_1\|_{W_0^{2,p}(0,\omega)} + \|F_2\|_{L^p(0,\omega)}.
\end{equation}
To this end, we consider that $\lambda \in (-\infty, - \varepsilon_0)$, where $\varepsilon_0$ is defined in \refP{Prop A inversible}. We first study $\psi_1''$. From \eqref{Sp}, for $a.e.$ $\theta \in [0,\omega]$, we have
\begin{equation*}
\begin{array}{rcl}
S''(\theta) & = & \dis \frac{\alpha_2 \,e^{-\theta \alpha_2}}{2\left(1 - e^{- 2\omega \alpha_2} \right)}  \left(J(0) - e^{- \omega \alpha_2} J(\omega)\right) - \frac{\lambda}{\alpha_1^2\alpha_2^2} F_1''(\theta) \\ \ecart 
&& \dis + \frac{\alpha_2\,e^{-(\omega-\theta) \alpha_2}}{2\left(1 - e^{- 2\omega \alpha_2} \right)} \left( J(\omega) - e^{- \omega \alpha_2} J(0)\right) - \frac{1}{2\alpha_2} J''(\theta),
\end{array}
\end{equation*}
and from \eqref{J}, we obtain $J''(\theta) = \alpha_2^2 J(\theta) - 2\alpha_2 v(\theta)$, hence
\begin{equation*}
\begin{array}{rcl}
S''(\theta) & = &  \dis \frac{\alpha_2 \,e^{-\theta \alpha_2}}{2\left(1 - e^{- 2\omega \alpha_2} \right)}  \left(J(0) - e^{- \omega \alpha_2} J(\omega)\right) - \frac{\lambda}{\alpha_1^2\alpha_2^2} F_1''(\theta) \\ \ecart 
&& \dis + \frac{\alpha_2\,e^{-(\omega-\theta) \alpha_2}}{2\left(1 - e^{- 2\omega \alpha_2} \right)} \left( J(\omega) - e^{- \omega \alpha_2} J(0)\right) - \frac{\alpha_2}{2} \, J(\theta) + v(\theta),
\end{array}
\end{equation*}
Then, since $\alpha_1 = \sqrt{-\lambda} + i$ and $\alpha_2 = \sqrt{-\lambda} - i$, we have $\left|e^{-\omega\alpha_1}\right| = \left|e^{-\omega\alpha_2}\right| = e^{-\omega\sqrt{-\lambda}} \leqslant 1$ with $-\lambda \geqslant \varepsilon_0$, thus 
\begin{equation*}
\begin{array}{rcl}
\|S''\|_{L^p(0,\omega)} & \leqslant & \dis \frac{\sqrt{1-\lambda} \, \left(|J(0)| + |J(\omega)|\right) }{2\left(1 - e^{- 2\omega \sqrt{\varepsilon_0}} \right)} \left( \int_0^\omega e^{-p\theta \sqrt{-\lambda}}~d\theta\right)^{1/p} \\ \ecart
&& \dis + \frac{\sqrt{1-\lambda} \,\left(|J(0)| + |J(\omega)|\right) }{2\left(1 - e^{- 2\omega \sqrt{\varepsilon_0}} \right)} \left( \int_0^\omega e^{-p(\omega-\theta) \sqrt{-\lambda}}~d\theta\right)^{1/p} \\ \ecart
&& \dis + \frac{-\lambda}{(1-\lambda)^2} \, \|F_1''\|_{L^p(0,\omega)} + \frac{\sqrt{1-\lambda}}{2}\, \|J\|_{L^p(0,\omega)} + \|v\|_{L^p(0,\omega)} \\ \\

& \leqslant & \dis \frac{\sqrt{1-\lambda} \,\left(|J(0)| + |J(\omega)|\right)}{\sqrt{-\lambda}^{1/p} \left(1 - e^{- 2\omega \sqrt{\varepsilon_0}} \right)}  + \frac{1}{1-\lambda} \, \|F_1''\|_{L^p(0,\omega)} \\ \ecart 
&& \dis + \frac{\sqrt{1-\lambda}}{2}\, \|J\|_{L^p(0,\omega)} + \|v\|_{L^p(0,\omega)}, 
\end{array}
\end{equation*}
From \refL{Lem estim J, I, v}, we have
\begin{equation*}
\begin{array}{rcl}
\|S''\|_{L^p(0,\omega)} & \leqslant & \dis \frac{2\sqrt{1-\lambda}}{\sqrt{-\lambda}^{1-1/p+1/p}\, \left(1 - e^{- 2\omega \sqrt{\varepsilon_0}} \right)} \,\|v\|_{L^p(0,\omega)} \\ \ecart
&& \dis + \frac{\sqrt{1-\lambda}}{\sqrt{-\lambda}}\, \|v\|_{L^p(0,\omega)} + \|v\|_{L^p(0,\omega)} + \frac{1}{1-\lambda} \, \|F_1''\|_{L^p(0,\omega)} \\ \\

& \leqslant & \dis \frac{2M_1}{-\lambda \left(1 - e^{- 2\omega \sqrt{\varepsilon_0}} \right)} \, \left(\|F_2\|_{L^p(0,\omega)} + 2 \|F_1\|_{L^p(0,\omega)} + 3 \|F_1''\|_{L^p(0,\omega)}\right) \\ \ecart
&& \dis + \frac{2M_1}{-\lambda} \left(\|F_2\|_{L^p(0,\omega)} + 2 \|F_1\|_{L^p(0,\omega)} + 3 \|F_1''\|_{L^p(0,\omega)}\right) + \frac{1}{1-\lambda} \, \|F_1''\|_{L^p(0,\omega)}.
\end{array}
\end{equation*}
Finally, we obtain
\begin{equation}\label{norme(Sp'')}
\|S''\|_{L^p(0,\omega)} \leqslant \frac{M_2}{-\lambda} \, \left(\|F_2\|_{L^p(0,\omega)} + 2 \|F_1\|_{L^p(0,\omega)} + 3 \|F_1''\|_{L^p(0,\omega)}\right),
\end{equation}
where $M_2 = \frac{2M_1}{1 - e^{- 2\omega \sqrt{\varepsilon_0}}} + 2 M_1 +1$.

Now, we consider $\psi_1''- S''$ which reads as
\begin{equation*}
\begin{array}{rcl}
\psi_1''(\theta) - S''(\theta) & = &\dis  \alpha_2^2\,e^{-\theta \alpha_2} (\beta_1 + \beta_2 + \beta_3 + \beta_4) +  \alpha_2^2\,e^{-(\omega-\theta)\alpha_2} (\beta_3 + \beta_4 - \beta_1 - \beta_2) \\ \ecart
&& \dis +  \alpha_2^2\,\left(e^{-\theta \alpha_1} - e^{-\theta \alpha_2}\right) (\beta_2 + \beta_4) \\ \ecart
&& \dis + \alpha_2^2\,\left(e^{-(\omega-\theta) \alpha_1} - e^{-(\omega-\theta) \alpha_2}\right) (\beta_4 - \beta_2),
\end{array}
\end{equation*}
where $\beta_i$, $i=1,2,3,4$ are given by \eqref{beta i} and \eqref{U-V M}. Then, we have
\begin{equation*}
\begin{array}{rcl}
\|\psi_1'' - S''\|_{L^p(0,\omega)} & \leqslant & \dis (1-\lambda) (|\beta_1+\beta_2|+|\beta_3+\beta_4|) \left(\int_0^\omega e^{-p\theta\sqrt{-\lambda}}~d\theta\right)^{1/p} \\ \ecart
&& \dis + (1-\lambda) (|\beta_1 + \beta_2|+|\beta_3 + \beta_4|) \left(\int_0^\omega e^{-p(\omega-\theta)\sqrt{-\lambda}}~d\theta\right)^{1/p} \\ \ecart
&& \dis + (1-\lambda) (|\beta_2|+|\beta_4|) \left(\int_0^\omega \left|e^{-\theta\alpha_1} - e^{-\theta\alpha_2}\right|^p ~d\theta\right)^{1/p} \\ \ecart
&& \dis + (1-\lambda) (|\beta_2|+|\beta_4|) \left(\int_0^\omega \left|e^{-(\omega-\theta)\alpha_1} - e^{-(\omega-\theta)\alpha_2}\right|^p~d\theta\right)^{1/p}.
\end{array}
\end{equation*}
Since $-\lambda \geqslant \varepsilon_0 > 0$, we have
$$\frac{1-\lambda}{-\lambda} = 1 + \frac{1}{-\lambda} \leqslant 1 + \frac{1}{\varepsilon_0},$$
and thus, from \refL{Lem estim norm Lp difference exp} and \refL{Lem estimation beta i}, we obtain 
$$\begin{array}{lll}
\|\psi_1'' - S''\|_{L^p(0,\omega)} & \leqslant & \dis\frac{2(1-\lambda) \left(|\beta_1 + \beta_2| + |\beta_3 + \beta_4| \right)}{\sqrt{-\lambda}^{1/p}} + \frac{8 (1-\lambda) (|\beta_2|+|\beta_4|)}{\sqrt{-\lambda}^{1+1/p}} \\ \ecart

& \leqslant & \dis \frac{M_3 \left(\|F_2\|_{L^p(0,\omega)} + 2 \|F_1\|_{L^p(0,\omega)} + 3 \|F_1''\|_{L^p(0,\omega)}\right)}{-\lambda},
\end{array}$$
where $M_3 = \frac{4 M_1 \left(1+\frac{1}{\omega}\right)\left(1+\frac{1}{\varepsilon_0}\right)}{\left(1 - e^{-2\omega \sqrt{\varepsilon_0}} - 2\omega \sqrt{\varepsilon_0}\, e^{-\omega \sqrt{\varepsilon_0}}\right) \left(1 - e^{-\omega \sqrt{\varepsilon_0}}\right)}$.

From \eqref{norme(Sp'')}, it follows that
$$\begin{array}{rcl}
\|\psi_1''\|_{L^p(0,\omega)} & \leqslant & \dis \|\psi_1'' - S''\|_{L^p(0,\omega)} + \|S''\|_{L^p(0,\omega)} \\ \ecart

& \leqslant & \dis \frac{M_3 + M_2}{-\lambda}\left(\|F_2\|_{L^p(0,\omega)} + 2 \|F_1\|_{L^p(0,\omega)} + 3 \|F_1''\|_{L^p(0,\omega)}\right) \\ \ecart

& \leqslant & \dis \frac{3(M_2 + M_3)}{-\lambda}\|F\|_{X}.
\end{array}$$ 
From the Poincar\'e inequality, there exists $C_\omega > 0$ such that
\begin{equation}\label{Norme W2p psi 1}
\|\psi_1\|_{W_0^{2,p}(0,\omega)} \leqslant C_\omega \|\psi_1''\|_{L^p(0,\omega)} \leqslant \frac{3 C_\omega (M_2 + M_3)}{-\lambda} \|F\|_X.
\end{equation}
Now, we focus ourselves on $\|\psi_2\|_{L^p(0,\omega)}$. As previously, by \eqref{psi2}, using \refL{Lem estim norm Lp difference exp} and \refL{Lem estimation beta i}, we obtain
\begin{equation*}
\begin{array}{lll}
\|\psi_2\|_{L^p(0,\omega)} & \leqslant & \dis |\lambda|  (|\beta_1 + \beta_2| + |\beta_3 + \beta_4|) \left(\int_0^\omega e^{-p\theta\sqrt{-\lambda}}~d\theta\right)^{1/p}\\ \ecart
&& \dis + |\lambda|  (|\beta_1 + \beta_2| + |\beta_3 + \beta_4|) \left(\int_0^\omega e^{-p(\omega-\theta)\sqrt{-\lambda}}~d\theta\right)^{1/p} \\ \ecart
&& \dis + |\lambda| (|\beta_2| + |\beta_4|) \left(\int_0^\omega \left|e^{-\theta\alpha_1} - e^{-\theta\alpha_2}\right|^p~d\theta\right)^{1/p} \\ \ecart
&&\dis + |\lambda| (|\beta_4| + |\beta_2|) \left(\int_0^\omega \left|e^{-(\omega-\theta)\alpha_1} - e^{-(\omega-\theta)\alpha_2}\right|^p~d\theta\right)^{1/p} \\ \ecart
&& \dis + \left\|\lambda S + F_1\right\|_{L^p(0,\omega)} \\ \\

& \leqslant & \dis \frac{M_2 + M_3}{-\lambda}\left(\|F_2\|_{L^p(0,\omega)} + 2 \|F_1\|_{L^p(0,\omega)} + 3 \|F_1''\|_{L^p(0,\omega)}\right) \\ \ecart
&& \dis + \left\|\lambda S + F_1\right\|_{L^p(0,\omega)}.
\end{array}
\end{equation*}
Moreover, from \eqref{lambda Sp + F1} and \refL{Lem estim J, I, v}, we deduce
\begin{equation*}
\begin{array}{rcl}
\left\|\lambda S + F_1\right\|_{L^p(0,\omega)} & \leqslant & \dis \frac{-\lambda \left(|J(0)| + |J(\omega)|\right)}{2 \sqrt{1-\lambda}\, \left(1 - e^{- 2\omega \sqrt{\varepsilon_0}} \right)} \,\left(\int_0^\omega e^{-p\theta \sqrt{-\lambda}} ~d\theta\right)^{1/p} \\ \ecart 
&& \dis + \frac{-\lambda \left(|J(0)| + |J(\omega)|\right)}{2 \sqrt{1-\lambda}\, \left(1 - e^{- 2\omega \sqrt{\varepsilon_0}} \right)} \,\left(\int_0^\omega e^{-p(\omega-\theta) \sqrt{-\lambda}} ~d\theta\right)^{1/p} \\ \ecart
&& \dis + \left|1 - \frac{\lambda^2}{\alpha_1^2\alpha_2^2}\right|\, \|F_1\|_{L^p(0,\omega)} + \frac{-\lambda}{2\sqrt{1-\lambda}}\, \|J\|_{L^p(0,\omega)} \\ \\

& \leqslant & \dis \frac{-\lambda \left(|J(0)| + |J(\omega)|\right)}{\sqrt{-\lambda}^{1+1/p} \left(1 - e^{- 2\omega \sqrt{\varepsilon_0}} \right)} + \left|1 - \frac{\lambda^2}{(1-\lambda)^2}\right|\, \|F_1\|_{L^p(0,\omega)} +  \|v\|_{L^p(0,\omega)} \\ \\

& \leqslant & \dis \frac{-2\lambda}{\sqrt{-\lambda}^{2}\, \left(1 - e^{- 2\omega \sqrt{\varepsilon_0}} \right)} \, \|v\|_{L^p(0,\omega)} + \frac{1-2\lambda}{(1-\lambda)^2}\, \|F_1\|_{L^p(0,\omega)} + \|v\|_{L^p(0,\omega)} \\ \\

& \leqslant & \dis \left(\frac{2}{1 - e^{- 2\omega \sqrt{\varepsilon_0}}} +1 \right)\, \|v\|_{L^p(0,\omega)} + \left(\frac{1}{\lambda^2}+ \frac{1}{-\lambda}\right)\, \|F_1\|_{L^p(0,\omega)} \\ \\

& \leqslant & \dis \left(\frac{2}{1 - e^{- 2\omega \sqrt{\varepsilon_0}}} +1 \right)\, \|v\|_{L^p(0,\omega)} + \frac{\frac{1}{\varepsilon_0}+ 1}{-\lambda} \|F_1\|_{L^p(0,\omega)}.
\end{array}
\end{equation*}
Then, from \refL{Lem estim J, I, v}, we obtain
$$\left\|\lambda S + F_1\right\|_{L^p(0,\omega)} \leqslant \frac{M_4}{-\lambda} \left(\|F_2\|_{L^p(0,\omega)} + 2 \|F_1\|_{L^p(0,\omega)} + 3 \|F_1''\|_{L^p(0,\omega)}\right),$$
where $M_4 =  \left(\frac{2}{1 - e^{- 2\omega \sqrt{\varepsilon_0}}} + 1 \right) M_1 + \frac{1}{\varepsilon_0} + 1$.

Thus, it follows that
$$\begin{array}{lll}
\|\psi_2\|_{L^p(0,\omega)} & \leqslant & \dis \frac{M_2+M_3+M_4}{-\lambda} \left(\|F_2\|_{L^p(0,\omega)} + 2 \|F_1\|_{L^p(0,\omega)} + 3 \|F_1''\|_{L^p(0,\omega)}\right) \\ \ecart
& \leqslant & \dis\frac{3\left(M_2+M_3+M_4\right)}{-\lambda} \|F\|_X.
\end{array}$$
Finally, from \eqref{Norme W2p psi 1}, we have
$$\|(\A-\lambda I)^{-1} F\|_X = \|\psi_1\|_{W^{2,p}_0(0,\omega)} + \|\psi_2\|_{L^p(0,\omega)} \leqslant \frac{M}{|\lambda|} \|F\|_X,$$
where $M = 3 \left((C_\omega + 1)(M_2+M_3) + M_4\right)$, which gives the result since $0\in\rho(\A)$ from \refP{Prop A inversible}. 
\end{proof}
Since $-\A$ is the realization of $\L_2$, we deduce the following corollary. 
\begin{Cor}\label{Cor L2}
There exist $\varepsilon_{\L_2} \in (0, \pi)$ small enough and $M_{\L_2} > 0$ such that
$$\forall ~z \in \Sigma_{\L_2} := \overline{B(0,\varepsilon_0)} \cup \{z \in \CC\setminus\{0\} : |\arg(z)| \leqslant \varepsilon_{\L_2}\},$$
we have
$$\left\|\left(\L_2 - z I\right)^{-1}\right\|_{\L(X)} \leqslant \frac{M_{\L_2}}{1+|z|}.$$
Therefore, assumption $(H_1)$ in \refS{sect M1 M2} is verified for $\L_2$ with
\begin{equation}\label{theta L2}
\theta_{\L_2} = \pi - \varepsilon_{\L_2}.
\end{equation}
\end{Cor}

\begin{Rem}\label{Rem valeurs propres L2}
$\A$ is anti-compact; since $\sigma(-\L_2) = \sigma(\A)$ then $\sigma(-\L_2)$ is uniquely composed by isolated eigenvalues $\left( \lambda _{j}\right) _{j\geqslant 1}$ such that $\left\vert \lambda _{j}\right\vert \rightarrow +\infty$, see \cite{kato}, Theorem 6.29, p. 187. More precisely, the calculus of the resolvent operator $\left(\A - \lambda I\right)^{-1}$ requires that, for all $\lambda \in \CC \setminus \RR_+$, $U_-$ and $U_+$ defined by \eqref{U-V M} do not vanish. Since $U_-U_+ = 0$ is equivalent to
$$\left(\sinh(\omega \sqrt{-\lambda}) - \omega \sqrt{-\lambda}\right)\left( \sinh(\omega \sqrt{-\lambda}) + \omega \sqrt{-\lambda}\right) = 0,$$
then, using $(z_j)_{j\geqslant 1}$ defined in \refS{Sect Intro}, we deduce that
$$\forall\, j \geqslant 1, \quad \lambda_j = -\frac{z_j^2}{\omega^2} \in \CC \setminus \RR_+.$$
\end{Rem}

Now, we prove that operator $\A$ has Bounded Imaginary Powers, see \refD{Def BIP}.
\begin{Prop}\label{Prop A BIP}
$\A \in$ BIP\,$(X,\theta_{\A})$, for any $\theta_{\A} \in (0, \pi)$.
\end{Prop}
\begin{proof}
This proof follow the same steps than those in the proofs of Proposition 3.1 in \cite{labbas-moussaoui} or Proposition 4.1 in \cite{labbas-sadallah}. From \refP{Prop estim resolvante A}, we deduce that $\A$ is a sectorial operator and from \refP{Prop A inversible}, we have $0 \in \rho(\A)$. Moreover, since $X=W_{0}^{2,p}(0,\omega)\times L^{p}(0,\omega)$ is a reflexive space, from \cite{haase}, Proposition 2.1.1, h), statement $i)$ of \refD{Def BIP} holds with $T_2 = \A$.

For all $\lambda > 0$, we have
$$\A \left( \begin{array}{c}
\psi _{1} \\ 
\psi _{2}
\end{array}\right) + \lambda \left( \begin{array}{c}
\psi _{1} \\ 
\psi _{2}
\end{array}\right) = \left( \begin{array}{c}
F_{1} \\ 
F_{2}
\end{array}\right),$$
which writes
\begin{equation}\label{Syst BIP}
\left\{ \begin{array}{l}
\psi_2 = F_1 - \lambda \psi_1 \\ \ecart
\psi _{1}^{(4)} + 2(1 - \lambda) \psi''_{1} - (\lambda + 1)^{2}\psi_{1} = G_\lambda \\ \ecart
\psi_{1}(0) = \psi_{1}(\omega) = \psi'_{1}(0) = \psi'_{1}(\omega) = 0,
\end{array}\right.
\end{equation}
where 
$$G_\lambda = - F_{2} - 2 (F''_{1}-F_1) + \lambda F_1.$$
From \refP{Prop sol psi 1}, the explicit expression of the solution to problem \eqref{Syst BIP} is
$$\left( \begin{array}{c}
\psi_{1} \\ 
\psi_{2}
\end{array}\right) = \left(\A + \lambda I\right)^{-1}\left( \begin{array}{c}
F_{1} \\ 
F_{2}
\end{array}\right).$$
Let $\varepsilon >0$ and $r \in \RR$. We have 
$$\left[\left(\A + I\right)^{-\varepsilon + ir} \left(\begin{array}{c}
F_1 \\
F_2
\end{array} \right)\right](\theta) = \frac{1}{\Gamma_{\varepsilon,r}}\int_0^{+\infty} \lambda^{-\varepsilon+ir}\left[\left(\A + I + \lambda I\right)^{-1}\left(\begin{array}{c}
F_1 \\
F_2
\end{array} \right)\right](\theta)d\lambda,$$
where $\Gamma_{\varepsilon,r} = \Gamma(1-\varepsilon + ir)\Gamma(\varepsilon-ir)$, see for instance \cite{triebel}, $(6)$, p. 100.

Now, we only focus on the first component $\psi_1$ since calculations are similar for $\psi_2$. Moreover, from \cite{seeley}, we only have to study the convolution term, which is the most singular term. In our case, this term is given by  
$$\begin{array}{lll}
I_{S}(\theta) & = & \dis\frac{1}{4\alpha_1\alpha_2}\int_0^\omega e^{-|\theta-s|\alpha_2}\int_0^\omega e^{-|s-t|\alpha_1} G_{\lambda+1}(t)~ dtds \\ \ecart

& = & \dis \frac{1}{4\alpha_1\alpha_2}\int_0^\omega e^{-|\theta-s|\alpha_2}\int_0^\omega e^{-|s-t|\alpha_1} \left(-F_2 - 2 (F''_1 - F_1) + (\lambda + 1) F_1\right)(t)~ dtds.
\end{array}$$
Note that, in order to simplify calculations, we have used the convolutions terms given in \eqref{Sp M} and \eqref{v0 M}. Moreover, since $F_1(0) = F_1''(0) = F_2(0) = F_1(\omega) = F_1''(\omega) = F_2(\omega) = 0$, we wet
$$\widetilde{G_0}(x) = \left\{\begin{array}{ll}
-F_2(x) - 2F_1''(x) + 2F_1(x), &\text{if } x \in [0,\omega] \\
0,    &\text{else},
\end{array}\right. \quad \text{with} \quad \widetilde{F_1}(x) = \left\{\begin{array}{ll}
F_1(x), &\text{if } x \in [0,\omega] \\
0,    &\text{else},
\end{array}\right.$$
and 
$$E_\alpha (x) = e^{-|x| \alpha}.$$
It follows that
$$\begin{array}{lll}
I_{\varepsilon,r} (\theta) & = & \dis \frac{1}{\Gamma_{\varepsilon,r}}\int_0^{+\infty} \lambda^{-\varepsilon + ir} I_S (\theta)~ d\lambda \\ \\

& = & \dis \frac{1}{\Gamma_{\varepsilon,r}}\int_0^{+\infty} \frac{\lambda^{-\varepsilon + ir}}{4\alpha_1\alpha_2} \left(E_{\alpha_2} \star \left(E_{\alpha_1}\star \left(\widetilde{G_0}+(\lambda+1)\widetilde{F_1}\right)\right)\right)(\theta)~ d\lambda \\ \\

& = & \dis \frac{1}{\Gamma_{\varepsilon,r}}\int_0^{+\infty} \frac{\lambda^{-\varepsilon + ir}}{4\alpha_1\alpha_2} \F^{-1}\left(\F\left(E_{\alpha_2} \star \left(E_{\alpha_1}\star \left(\widetilde{G_0}+(\lambda+1)\widetilde{F_1}\right)\right)\right)(\xi) \right)(\theta)~ d\lambda \\ \\

& = & \dis \F^{-1}\left(\frac{1}{\Gamma_{\varepsilon,r}}\int_0^{+\infty} \frac{\lambda^{-\varepsilon + ir}}{4\alpha_1\alpha_2} \F\left(E_{\alpha_2}\right)(\xi) \F\left(E_{\alpha_1}\right)(\xi) \F\left(\widetilde{G_0}+(\lambda+1)\widetilde{F_1}\right)(\xi)~ d\lambda\right)(\theta) \\ \\

& = & \dis \F^{-1}\left(\frac{1}{\Gamma_{\varepsilon,r}}\int_0^{+\infty} \frac{\lambda^{-\varepsilon + ir}}{4\alpha_1\alpha_2} \F\left(E_{\alpha_2}\right)(\xi) \F\left(E_{\alpha_1}\right)(\xi) ~ d\lambda \, \F\left(\widetilde{G_0}\right)(\xi)\right)(\theta) \\ \ecart
&& \dis + \F^{-1}\left(\frac{1}{\Gamma_{\varepsilon,r}}\int_0^{+\infty} \frac{\lambda^{-\varepsilon + ir}(\lambda+1)}{4\alpha_1\alpha_2} \F\left(E_{\alpha_2}\right)(\xi) \F\left(E_{\alpha_1}\right)(\xi)~ d\lambda \, \F\left(\widetilde{F_1}\right)(\xi)\right)(\theta).
\end{array}$$
We recall that 
$$\F\left(E_\alpha\right)(\xi) = \frac{2\alpha}{\alpha^2 + 4 \pi^2\xi^2},$$
and here $\alpha_1 = \sqrt{\lambda + 1} + i$ and $\alpha_2 = \sqrt{\lambda + 1} - i$. Hence
$$\begin{array}{lll}
\dis \frac{\lambda^{-\varepsilon + ir}}{4\alpha_1\alpha_2} \F\left(E_{\alpha_2}\right)(\xi) \F\left(E_{\alpha_1}\right)(\xi) & = & \dis \frac{\lambda^{-\varepsilon + ir}}{4\alpha_1\alpha_2} \frac{4 \alpha_1 \alpha_2}{(\alpha_1^2 + 4\pi^2\xi^2)(\alpha_2^2 + 4\pi^2\xi^2)} \\ \ecart
& = & \dis\frac{\lambda^{-\varepsilon + ir}}{\alpha_1^2\alpha_2^2 + 4\pi^2\xi^2(\alpha_1^2 + \alpha_2^2) + 16\pi^4\xi^4} \\ \ecart

& = & \dis\frac{\lambda^{-\varepsilon + ir}}{\lambda^2 + 4(1+2\pi^2\xi^2)\lambda + 4(1 + 4 \pi^4\xi^4)} \\ \ecart

& = & \dis\frac{\lambda^{-\varepsilon + ir}}{(\lambda + \lambda_1)(\lambda + \lambda_2)},
\end{array}$$
where
$$\left\{\begin{array}{l}
\dis\lambda_1 = 2 + 4\pi\xi + 4 \pi^2\xi^2 = 4\pi^2\left(\xi - \frac{(1+i)}{2\pi}\right) \left(\xi - \frac{(1-i)}{2\pi}\right) \\ \ecart

\dis\lambda_2 = 2 - 4\pi\xi + 4 \pi^2\xi^2 = 4\pi^2\left(\xi + \frac{(1+i)}{2\pi}\right) \left(\xi + \frac{(1-i)}{2\pi}\right).
\end{array}\right.$$
Thus, since we have
$$\frac{1}{(\lambda + \lambda_1)(\lambda + \lambda_2)} = \frac{1}{\lambda_1 - \lambda_2} \left(- \frac{1}{\lambda + \lambda_1} + \frac{1}{\lambda + \lambda_2}\right),$$
it follows that
$$\frac{\lambda^{-\varepsilon + ir}}{4\alpha_1\alpha_2} \F\left(E_{\alpha_2}\right)(\xi) \F\left(E_{\alpha_1}\right)(\xi) = \frac{1}{8\pi\xi} \left(- \frac{\lambda^{-\varepsilon + ir}}{\lambda + \lambda_1} + \frac{\lambda^{-\varepsilon + ir}}{\lambda + \lambda_2}\right).$$
Then, setting 
$$\sigma_1 = \frac{\lambda}{\lambda_1} \quad \text{and} \quad \sigma_2 = \frac{\lambda}{\lambda_2},$$
we obtain
$$\begin{array}{lll}
\dis\int_0^{+\infty} \frac{\lambda^{-\varepsilon + ir}}{4\alpha_1\alpha_2} \F\left(E_{\alpha_2}\right)(\xi) \F\left(E_{\alpha_1}\right)(\xi)~ d\lambda & = & \dis \frac{1}{8\pi\xi}\left(-\int_0^{+\infty}  \frac{\lambda^{-\varepsilon + ir}}{\lambda + \lambda_1} ~d\lambda + \int_0^{+\infty}  \frac{\lambda^{-\varepsilon + ir}}{\lambda + \lambda_2} ~d\lambda \right) \\ \\

& = & \dis -\frac{\lambda_1^{-\varepsilon + ir}}{8\pi\xi} \int_0^{+\infty}  \frac{\sigma_1^{-\varepsilon + ir}}{\sigma_1 + 1} ~d\sigma_1 \\ \ecart
&& \dis + \frac{\lambda_2^{-\varepsilon + ir}}{8\pi\xi}\int_0^{+\infty}  \frac{\sigma_2^{-\varepsilon + ir}}{\sigma_2 + 1} ~d\sigma_2. 
\end{array}$$
Moreover, for all $z \in \CC \setminus \NN^-$, where $\NN^-$ is the set of negative integer, we have
\begin{equation}\label{int=Gamma}
\int_0^{+\infty}  \frac{\sigma^{-z}}{\sigma + 1} ~d\sigma = \Gamma(z)\Gamma(1-z),
\end{equation}
thus, it follows that
\begin{equation}\label{int=Gamma2}
\int_0^{+\infty}  \frac{\sigma^{-\varepsilon + ir}}{\sigma + 1} ~d\sigma = \Gamma(\varepsilon - ir)\Gamma(1-\varepsilon + ir) = \Gamma_{\varepsilon,r},
\end{equation}
hence
$$\begin{array}{lll}
\dis\frac{1}{\Gamma_{\varepsilon,r}}\int_0^{+\infty} \frac{\lambda^{-\varepsilon + ir}}{4\alpha_1\alpha_2} \F\left(E_{\alpha_2}\right)(\xi) \F\left(E_{\alpha_1}\right)(\xi)~ d\lambda & = & \dis\frac{1}{8\pi\xi \,\Gamma_{\varepsilon,r}}\left(\lambda_2^{-\varepsilon + ir} - \lambda_1^{-\varepsilon + ir}\right) \Gamma_{\varepsilon,r} \\ \ecart

& = & \dis\frac{1}{8\pi\xi}\left(\lambda_2^{-\varepsilon + ir} - \lambda_1^{-\varepsilon + ir}\right).
\end{array}$$
In the same way, we have
$$\begin{array}{lll}
\dis \frac{\lambda^{-\varepsilon + ir}(\lambda+1)}{4\alpha_1\alpha_2} \F\left(E_{\alpha_2}\right)(\xi) \F\left(E_{\alpha_1}\right)(\xi) & = & \dis \frac{1}{\lambda_1 - \lambda_2} \left(-\frac{(\lambda + 1)\lambda^{-\varepsilon+ir}}{\lambda + \lambda_1} + \frac{(\lambda + 1)\lambda^{-\varepsilon+ir}}{\lambda + \lambda_2}\right) \\ \\

& = & \dis  \frac{1}{8\pi \xi} \left(-\frac{\lambda^{1-\varepsilon+ir}}{\lambda + \lambda_1} + \frac{\lambda^{1-\varepsilon+ir}}{\lambda + \lambda_2}\right) \\ \ecart
&&\dis  + \frac{1}{8\pi\xi} \left(-\frac{\lambda^{-\varepsilon+ir}}{\lambda + \lambda_1} + \frac{\lambda^{-\varepsilon+ir}}{\lambda + \lambda_2}\right).
\end{array}$$
Then, setting 
$$\sigma_1 = \frac{\lambda}{\lambda_1} \quad \text{and} \quad \sigma_2 = \frac{\lambda}{\lambda_2},$$
we obtain
$$\begin{array}{lll}
\Upsilon & = & \dis\int_0^{+\infty} \frac{\lambda^{-\varepsilon + ir}(\lambda+1)}{4\alpha_1\alpha_2} \F\left(E_{\alpha_2}\right)(\xi) \F\left(E_{\alpha_1}\right)(\xi)~ d\lambda \\ \\

& = & \dis -\frac{1}{8\pi\xi}\int_0^{+\infty}  \frac{\lambda^{-\varepsilon + ir}}{\lambda + \lambda_1} ~d\lambda + \frac{1}{8\pi\xi}\int_0^{+\infty}  \frac{\lambda^{-\varepsilon + ir}}{\lambda + \lambda_2} ~d\lambda \\ \ecart
&& \dis - \frac{1}{8\pi\xi}\int_0^{+\infty}  \frac{\lambda^{1-\varepsilon + ir}}{\lambda + \lambda_1} ~d\lambda  + \frac{1}{8\pi\xi}\int_0^{+\infty}  \frac{\lambda^{1-\varepsilon + ir}}{\lambda + \lambda_2} ~d\lambda \\ \\

& = & \dis -\frac{\lambda_1^{-\varepsilon + ir}}{8\pi\xi} \int_0^{+\infty}  \frac{\sigma_1^{-\varepsilon + ir}}{\sigma_1 + 1} ~d\sigma_1 + \frac{\lambda_2^{-\varepsilon + ir}}{8\pi\xi}\int_0^{+\infty}  \frac{\sigma_2^{-\varepsilon + ir}}{\sigma_2 + 1} ~d\sigma_2 \\ \ecart
&& \dis -\frac{\lambda_1^{1-\varepsilon + ir}}{8\pi\xi} \int_0^{+\infty}  \frac{\sigma_1^{1-\varepsilon + ir}}{\sigma_1 + 1} ~d\sigma_1 + \frac{\lambda_2^{1-\varepsilon + ir}}{8\pi\xi}\int_0^{+\infty}  \frac{\sigma_2^{1-\varepsilon + ir}}{\sigma_2 + 1} ~d\sigma_2.
\end{array}$$
Moreover, from \eqref{int=Gamma} and \eqref{int=Gamma2}, we deduce that
$$\begin{array}{lll}
\Upsilon & = & \dis \left(\frac{\lambda_2^{-\varepsilon + ir} - \lambda_1^{-\varepsilon + ir}}{8\pi\xi}\right)\Gamma_{\varepsilon,r} + \left(\frac{\lambda_2^{-\varepsilon + ir} - \lambda_1^{-\varepsilon + ir}}{8\pi\xi}\right) \Gamma(\varepsilon - ir - 1)\Gamma(1-(\varepsilon - ir -1)) \\ \\

& = & \dis \left(\frac{\lambda_2^{-\varepsilon + ir} - \lambda_1^{-\varepsilon + ir}}{2\pi\xi} \right) \left(\Gamma_{\varepsilon,r} + \Gamma(\varepsilon - ir - 1)\Gamma(1-(\varepsilon - ir -1))\right).
\end{array}$$
For all $z \in \CC \setminus \ZZ$, we have
$$\Gamma(z - 1)\Gamma(1-(z - 1)) = \frac{\pi}{\sin(\pi(z-1))} = -\frac{\pi}{\sin(\pi z)} = -\Gamma(z)\Gamma(1-z),$$
Setting $z=\varepsilon - ir$, with $\varepsilon \in (0,1)$, it follows that
$$\Gamma(\varepsilon - ir - 1)\Gamma(1-(\varepsilon - ir -1) = - \Gamma(\varepsilon - ir)\Gamma(1 - \varepsilon + ir) = - \Gamma_{\varepsilon,r},$$
hence $\Upsilon = 0$. Finally, we obtain that
$$I_{\varepsilon,r} (\theta) = \F^{-1}\left(m_\varepsilon (\xi)\F\left(\widetilde{G_0}\right)(\xi)\right)(\theta),$$
where
$$m_\varepsilon (\xi) = \frac{\lambda_2^{-\varepsilon + ir} - \lambda_1^{-\varepsilon + ir}}{8\pi\xi}.$$
Setting
$$m(\xi) := \lim_{\varepsilon \to 0} m_\varepsilon (\xi)  = \frac{\lambda_2^{ir} - \lambda_1^{ir}}{8\pi\xi},$$
due to the Lebesgue's dominated convergence Theorem, it follows that
$$I_{0,r} (\theta) := \lim_{\varepsilon \to 0} I_{\varepsilon,r} (\theta)  = \F^{-1}\left(m (\xi)\F\left(\widetilde{G_0}\right)(\xi)\right)(\theta).$$
Moreover, for all $x_1,x_2 \in \RR$, we have 
$$\left|e^{ix_1} - e^{ix_2}\right| \leqslant |x_1 - x_2|,$$
then, for all $\xi \in \RR \setminus\{0\}$, we deduce that
$$|m(\xi)| = \frac{\left|\lambda_2^{ir} - \lambda_1^{i r}\right|}{8\pi |\xi|} = \frac{\left|e^{ir\ln(\lambda_2)} - e^{i r \ln(\lambda_1)}\right|}{8\pi |\xi|} \leqslant \frac{|r| \left|\ln(\lambda_2) - \ln(\lambda_1)\right|}{8\pi|\xi|}.$$ 
Thus
$$\sup_{\xi \in \RR}\left|m (\xi) \right| = \lim_{\xi \to 0} \left|m (\xi) \right| = \left| \lim_{\xi \to 0} \frac{\lambda_2^{ir} - \lambda_1^{ir}}{8\pi\xi}\right|,$$
and 
$$\begin{array}{lll}
\dis\lim_{\xi \to 0} \frac{\lambda_2^{ir} - \lambda_1^{ir}}{8\pi\xi} & = & \dis \lim_{\xi \to 0} 2^{ir}\left(\frac{1 + ir\left(-2 \pi\xi + 2 \pi^2 \xi^2\right) + 2 ir(ir - 1)\pi^2\xi^2 + o(\xi^2)}{8\pi\xi}\right) \\ \ecart

&& \dis -  \lim_{\xi \to 0} 2^{ir}\left(\frac{1 + ir\left(2\pi\xi + 2 \pi^2\xi^2\right) + 2 ir(ir - 1)\pi^2\xi^2 + o(\xi^2)}{8\pi\xi}\right) \\ \\

& = & \dis \lim_{\xi \to 0} 2^{ir}\left( \frac{-4 ir\pi\xi + o(\xi^2)}{8\pi\xi} \right)\\ \\

& = & \dis -2^{ir-1}ir.
\end{array}$$
Then
$$\sup_{\xi \in \RR}\left|m (\xi) \right| = \frac{|r|}{2}.$$
Since we have
$$\begin{array}{lll}
\dis\xi \,m'(\xi) & = & \dis \frac{2\pi\xi^2\left(ir \lambda_2^{ir-1} (-4\pi + 8 \pi^2\xi) - ir \lambda_1^{ir - 1}(4\pi + 8\pi^2\xi)\right) - 2\pi \xi \left(\lambda_2^{ir} - \lambda_1^{ir}\right)}{64\pi^2\xi^2} \\ \ecart

& = & \dis \frac{ir}{32}\left( \lambda_2^{ir-1} (-1 + 2 \pi\xi) - \lambda_1^{ir - 1}(1 + 2\pi\xi)\right) - \frac{\lambda_2^{ir} - \lambda_1^{ir}}{32\pi\xi},
\end{array}$$
we obtain in the same way that
$$\sup_{\xi \in \RR}\left|\xi \, m'(\xi) \right| = \lim_{\xi \to 0} \left|\xi \,m'(\xi)\right| =  \left|\lim_{\xi \to 0}\xi \, m'(\xi)\right|,$$
with
$$\lim_{\xi \to 0} \xi \, m'(\xi) = \frac{ir}{32}\lim_{\xi \to 0} \left( \lambda_2^{ir-1} (-1 + 2 \pi\xi) - \lambda_1^{ir - 1}(1 + 2\pi\xi)\right) - \frac{\lambda_2^{ir} - \lambda_1^{ir}}{32\pi\xi},$$
where
$$\lim_{\xi \to 0} \lambda_2^{ir-1} = \lim_{\xi \to 0} 2^{ir-1}\left(1 + (ir - 1)\left(-2 \pi\xi + 2 \pi^2 \xi^2\right) + 4  ir(ir - 1)\pi^2\xi^2 + o(\xi^2)\right) = 2^{ir-1},$$
and
$$\lim_{\xi \to 0} \lambda_1^{ir-1} = \lim_{\xi \to 0} 2^{ir-1}\left(1 + (ir - 1)\left(2 \pi\xi + 2 \pi^2 \xi^2\right) + 4 ir(ir - 1)\pi^2\xi^2 + o(\xi^2)\right) = 2^{ir-1}.$$
Thus, we obtain
$$\begin{array}{lll}
\dis \lim_{\xi \to 0} \xi \, m'(\xi) & = & \dis \frac{1}{4} \lim_{\xi \to 0} \frac{ir}{8}\left(-\left(\lambda_2^{ir-1} + \lambda_1^{ir - 1}\right) + 2 \pi \xi \left( \lambda_2^{ir-1} - \lambda_1^{ir - 1}\right)\right) - \frac{\lambda_2^{ir} - \lambda_1^{ir}}{8\pi\xi} \\ \ecart

& = & \dis \frac{1}{4}\left(- \frac{2^{ir}ir}{8} + 2^{ir-1}ir\right) \\ \ecart 

& = & \dis 3\times 2^{ir-5}ir.
\end{array}$$
Then
$$\sup_{\xi \in \RR}\left|\xi \,m'(\xi) \right| = \left|3\times 2^{ir-5}ir \right| = \frac{3}{32}|r|.$$
Therefore, we deduce that
$$\sup_{\xi \in \RR}\left|m (\xi) \right| + \sup_{\xi \in \RR}\left|\xi \,m'(\xi) \right| = \frac{|r|}{2} + \frac{3}{32}|r| = \frac{19}{32} |r|,$$
From the Mihlin Theorem, see \cite{mihlin}, for all $\gamma > 0$, there exists $C_{\gamma,p} >0$, such that, for all $r \in \RR$, we have
$$\left\|I_{0,r} (.)\right\|_{\L(X)} = \left\|F^{-1}\left(m (\xi) F(\widetilde{G_0})(\xi) \right)(.)\right\|_{\L(X)} \leqslant C_{\gamma,p} e^{\gamma |r|}.$$
We can treat the others terms using \cite{seeley} and obtain similar results. Finally, for all $\gamma > 0$, there exists a constant $C_{\gamma,p} > 0$ such that for all $r \in \RR$, we obtain
$$\left\|\left(\A + I\right)^{ir}\right\|_{\L(X)} \leqslant C_{\gamma,p} \,e^{\gamma |r|}.$$
Therefore, taking $\theta_\A = \gamma > 0$, we obtain that $\A + I \in $ BIP\,$(X,\theta_\A)$. Finally, using Theorem 2.3, p. 69 in \cite{arendt-bu-haase}, we deduce that $\A = \A+I-I \in $ BIP\,$(X,\theta_\A)$.
\end{proof}

\section{Study of the sum $\L_1+\L_2$} \label{Sect study of L1+L2}

\subsection{Invertibility of the closure of the sum}\label{sect inv sum}

In this section, we will apply the results described in \refS{sect M1 M2}. We take $\L_1 = \M_1$ and $\L_2 = \M_2$. 

\begin{Th}\label{Th L1+L2 barre inversible}
Assume that \eqref{hyp inv sum} holds. Then $\L_1 + \L_2$ is closable and its closure $\overline{\L_1 + \L_2}$ is invertible. 
\end{Th}
\begin{proof}
Assumption $(H_1)$ is satisfied from \refP{Prop L1} and \refC{Cor L2}, with 
$$\theta_{\M_1} + \theta_{\M_2} = \varepsilon_{\L_1} + \pi - \varepsilon_{\L_2},$$
where it suffices to take $\varepsilon_{\L_2} > \varepsilon_{\L_1}$ in order to obtain $\theta_{\M_1} + \theta_{\M_2} < \pi$. 

For assumption $(H_2)$, due to \refP{Prop A inversible}, it follows that $0 \notin \sigma(\L_1) \cap \sigma(-\L_2)$. Moreover, from \refP{Prop L1}, we have
$$\sigma(\L_1) = \{\lambda \in \CC : |\arg(\lambda)| < \pi \quad \text{and} \quad \text{Re}(\sqrt{\lambda}) \leqslant \nu\},$$
and from \refR{Rem valeurs propres L2}, it follows that
$$\sigma(-\L_2) = \{\lambda \in \CC\setminus \RR_+ : \sinh(\omega\sqrt{-\lambda}) = \pm \omega\sqrt{-\lambda}\} = \left\{-\frac{z_j^2}{\omega^2} \in \CC\setminus\RR_+: j \in \NN\setminus\{0\}\right\}.$$
Then, since
$$\text{Re}\left(\sqrt{-\frac{z_j^2}{\omega^2}}\right) = \frac{1}{\omega} \left|\text{Im}(z_j)\right|,$$
the condition $\sigma(\L_1) \cap \sigma(-\L_2) = \emptyset$ is fulfilled if \eqref{hyp inv sum} holds. 

The commutativity assumption $(H_3)$ is clearly verified since the actions of operators $\L_1$ and $\L_2$ are independent. 

Now, applying \refT{Th daprato-grisvard}, we obtain the result. 
\end{proof}

\begin{Rem}
We can conjecture that, for the critical case $\nu = \tau$, the sum $\L_1 + \L_2$ is not closable.
\end{Rem}

\subsection{Convexity inequalities}

In view to apply \refC{Cor inv fermeture somme}, we are going to verify inequality \eqref{ineg norme w} in two situations.

\begin{Prop}
Let
\begin{equation*}
\mathcal{E}_1=W^{1,p}(0,+\infty ;X)\subset \mathcal{E=}L^{p}(0,+\infty
;X),
\end{equation*}
and 
$$\E_{2} = L^{p}\left( 0,+\infty ;\left[ W^{3,p}(0,\omega)\cap W_{0}^{2,p}(0,\omega)\right] \times L^p(0,\omega)\right) \subset \mathcal{E}.$$
Then, we have
\begin{equation}\label{L1+L2 dans E1 inter E2}
D\left(\overline{\L_1 + \L_2}\right)\subset \E_{1} \cap \E_2.
\end{equation}
\end{Prop}

\begin{proof}
Let $V\in D(\mathcal{L}_{1})$. We must prove that there exists $\delta \in (0,1)$
such that 
\begin{equation*}
\left\Vert V\right\Vert _{\mathcal{E}_1}\leqslant C\left[ \left\Vert
V\right\Vert _{\mathcal{E}}+\left\Vert V\right\Vert _{\mathcal{E}}^{1-\delta
}\left\Vert \mathcal{L}_{1}(V)\right\Vert _{\mathcal{E}}^{\delta }\right].
\end{equation*}
For all $V\in W^{2,p}(0,+\infty ;X)$, from \cite{kato}, inequality (1.15), p. 192, we have the convexity inequality
\begin{equation*}
\|V'\|_{\E} \leqslant 2\sqrt{2} \|V\|_{\E}^{1/2} \|V''\|_{\E}^{1/2}.
\end{equation*}
Thus, we deduce
$$\|V\|_{\E_1} = \|V\|_\E + \|V'\|_\E \leqslant \|V\|_\E +  2\sqrt{2} \|V\|_{\E}^{1/2} \|V''\|_{\E}^{1/2}.$$
Since $\L_1$ is not invertible, we are going to estimate $\|V''\|_\E$ by $\|\L_1(V) - \lambda_0 V\|_\E$, where $\lambda_0 \in \rho(\L_1)$. We have 
$$V''-2\nu V'+(\nu^{2} - \lambda_0)V = \L_1(V) - \lambda_0 V.$$
Then, there exists a constant $C > 0$ such that
$$\|V''\|_\E + \|V'\|_\E + \|V\|_\E \leqslant C\|\L_1(V) - \lambda_0 V\|_\E,$$
hence
$$\|V''\|_\E \leqslant C \|\L_1(V) - \lambda_0 V\|_\E \leqslant C\|\L_1(V)\|_\E + |\lambda_0| C \|V\|_\E.$$
Thus, we deduce
$$\begin{array}{lll}
\dis \|V\|_{\E_1} = \|V\|_\E + \|V'\|_\E & \leqslant & \dis \|V\|_\E +  2\sqrt{2}\, \|V\|_{\E}^{1/2} \|V''\|_{\E}^{1/2} \\ \ecart
& \leqslant & \dis \dis \|V\|_\E +  2\sqrt{2C}\, \|V\|_{\E}^{1/2} \left(\|\L_1(V)\|_\E + |\lambda_0| \|V\|_\E\right)^{1/2} \\ \ecart
& \leqslant & \dis \dis \|V\|_\E +  2\sqrt{2C}\, \|V\|_{\E}^{1/2} \left(\|\L_1(V)\|_\E^{1/2} + |\lambda_0|^{1/2} \|V\|_\E^{1/2}\right) \\ \ecart
& \leqslant & \dis \dis \left(1 + 2\sqrt{2C}\,|\lambda_0|^{1/2} \right) \|V\|_\E +  2\sqrt{2C}\, \|V\|_{\E}^{1/2} \|\L_1(V)\|_\E^{1/2}.
\end{array}$$
Therefore, inequality \eqref{ineg norme w} is satisfied for $\delta = 1/2$ and $\M_1 = \L_1$. Using \refC{Cor inv fermeture somme}, we obtain 
\begin{equation*}
D\left(\overline{\L_1 + \L_2}\right)\subset \E_{1}.
\end{equation*}
Now, we must show that, for all $V\in D(\mathcal{L}_{2})$, we have
\begin{equation*}
\|V\|_{\E_2}\leqslant C \left[\|V\|_{\E} + \|V\|_{\E}^{1/2} \|\L_{2}(V)\|_{\E}^{1/2}\right].
\end{equation*}
To this end, it suffices to do it for $\mathcal{A}$. Set
\begin{equation*}
\mathcal{G}_{1}=\left[ W^{3,p}(0,\omega)\cap W_{0}^{2,p}(0,\omega)\right] \times L^p(0,\omega)\subset X.
\end{equation*}
We must prove that 
\begin{equation*}
\forall \left( 
\begin{array}{c}
\psi _{1} \\ 
\psi _{2}%
\end{array}%
\right) \in D(\mathcal{A}), \quad \left\Vert \left( 
\begin{array}{c}
\psi _{1} \\ 
\psi _{2}
\end{array}
\right) \right\Vert _{\mathcal{G}_{1}}\leqslant C\left[ \left\Vert \left( 
\begin{array}{c}
\psi _{1} \\ 
\psi _{2}
\end{array}
\right) \right\Vert _{X}+\left\Vert \left( 
\begin{array}{c}
\psi _{1} \\ 
\psi _{2}
\end{array}
\right) \right\Vert _{X}^{1/2}\left\Vert \mathcal{A}\left( 
\begin{array}{c}
\psi _{1} \\ 
\psi _{2}
\end{array}\right) \right\Vert _{X}^{1/2}\right].
\end{equation*}
Here, we have
$$\begin{array}{lll}
\left\|\left( \begin{array}{c}
\psi _{1} \\ 
\psi _{2}
\end{array}\right) \right\| _{\mathcal{G}_{1}} &=& \left\|\psi_1\right\|_{W^{3,p}(0,\omega)} + \left\|\psi_{2}\right\|_{L^p(0,\omega)} \\ \ecart
&=& \left\|\psi_1\right\|_{L^p(0,\omega)} + \left\|\psi'_1\right\|_{L^p(0,\omega)} + \left\|\psi''_1\right\|_{L^p(0,\omega)} + \left\|\psi'''_1\right\|_{L^p(0,\omega)} + \left\|\psi_2\right\|_{L^p(0,\omega)}.
\end{array}$$
Set $\varphi = \psi_1''$. Then, for all $\eta > 0$, from \cite{kato}, inequality (1.12), p. 192, taking $n=\eta+1$ and $b-a = \omega$, we obtain 
$$\|\varphi'\|_{L^p(0,\omega)} \leqslant \frac{\omega}{\eta} \|\varphi''\|_{L^p(0,\omega)} + \frac{2}{\omega} \left(\eta + 3 + \frac{2}{\eta}\right)\|\varphi\|_{L^p(0,\omega)}.$$
It is not difficult to see that the second member is minimal when 
$$\eta = \frac{\sqrt{2}}{2} \frac{\left(\|\varphi''\|_{L^p(0,\omega)} + 4 \|\varphi\|_{L^p(0,\omega)}\right)^{1/2}}{\|\varphi\|_{L^p(0,\omega)}^{1/2}}.$$
Therefore, we deduce that
$$\begin{array}{lll}
\|\varphi'\|_{L^p(0,\omega)} & \leqslant & \dis \frac{\omega \sqrt{2}\,\|\varphi\|_{L^p(0,\omega)}^{1/2}\|\varphi''\|_{L^p(0,\omega)}}{\left(\|\varphi''\|_{L^p(0,\omega)} + 4 \|\varphi\|_{L^p(0,\omega)}\right)^{1/2}} + \frac{4}{\omega} \frac{\sqrt{2}\,\|\varphi\|_{L^p(0,\omega)}^{1/2}\|\varphi\|_{L^p(0,\omega)} }{\left(\|\varphi''\|_{L^p(0,\omega)} + 4 \|\varphi\|_{L^p(0,\omega)}\right)^{1/2}} \\ \ecart 

&&\dis + \frac{\sqrt{2}}{\omega} \frac{\left(\|\varphi''\|_{L^p(0,\omega)} + 4 \|\varphi\|_{L^p(0,\omega)}\right)^{1/2}\|\varphi\|_{L^p(0,\omega)} }{\|\varphi\|_{L^p(0,\omega)}^{1/2}} + \frac{6}{\omega} \|\varphi\|_{L^p(0,\omega)} \\ \\

& \leqslant & \dis \frac{\sqrt{2}}{\omega} \left(\|\varphi''\|_{L^p(0,\omega)} + 4 \|\varphi\|_{L^p(0,\omega)}\right)^{1/2} \|\varphi\|_{L^p(0,\omega)}^{1/2} + \frac{6}{\omega} \|\varphi\|_{L^p(0,\omega)} \\ \ecart

&& \dis + \left(\frac{4}{\omega}\|\varphi\|_{L^p(0,\omega)} + \omega \|\varphi''\|_{L^p(0,\omega)}\right) \frac{\sqrt{2}\,\|\varphi\|_{L^p(0,\omega)}^{1/2}}{\left(\|\varphi''\|_{L^p(0,\omega)} + 4 \|\varphi\|_{L^p(0,\omega)}\right)^{1/2}} \\ \\

& \leqslant & \dis C_\omega \left(\|\varphi\|_{L^p(0,\omega)} + \|\varphi\|_{L^p(0,\omega)}^{1/2} \|\varphi''\|_{L^p(0,\omega)}^{1/2}\right).
\end{array}$$
Then, we have
$$\|\psi_1'''\|_{L^p(0,\omega)} \leqslant C_\omega \left(\|\psi_1''\|_{L^p(0,\omega)} + \|\psi_1''\|_{L^p(0,\omega)}^{1/2} \|\psi_1^{(4)}\|_{L^p(0,\omega)}^{1/2}\right).$$
Hence
$$\begin{array}{lll}
\left\|\left( \begin{array}{c}
\psi _{1} \\ 
\psi _{2}
\end{array}\right) \right\| _{\mathcal{G}_{1}} & \leqslant & \left\|\psi_1\right\|_{L^p(0,\omega)} + \left\|\psi'_1\right\|_{L^p(0,\omega)} + \left\|\psi''_1\right\|_{L^p(0,\omega)} \\ \ecart
&& \dis + C_\omega \left(\|\psi_1''\|_{L^p(0,\omega)} + \|\psi_1''\|_{L^p(0,\omega)}^{1/2} \|\psi_1^{(4)}\|_{L^p(0,\omega)}^{1/2}\right) + \left\|\psi_2\right\|_{L^p(0,\omega)} \\ \\

& \leqslant & \left(1 + C_\omega\right) \left\|\psi_1\right\|_{W^{2,p}_0(0,\omega)} + C_\omega  \|\psi_1\|_{W^{2,p}_0(0,\omega)}^{1/2} \|\psi_1^{(4)}\|_{L^p(0,\omega)}^{1/2} + \left\|\psi_2\right\|_{L^p(0,\omega)}
\end{array}$$
Now, in virtue of the invertibility of $\A$, see \refP{Prop A inversible}, we have proved that there exists a constant $C'_\omega$ depending only on $\omega$ such that
$$\|\psi_1^{(4)}\|_{L^p(0,\omega)} \leqslant C'_\omega \left\|\A\left(\begin{array}{c}
\psi_1 \\
\psi_2
\end{array}\right)\right\|_X.$$
Moreover, it follows
$$\begin{array}{lll}
\left\|\left( \begin{array}{c}
\psi _{1} \\ 
\psi _{2}
\end{array}\right) \right\| _{\mathcal{G}_{1}} & \leqslant & \left(1 + C_\omega\right) \left\|\left(\begin{array}{c}
\psi_1 \\
\psi_2
\end{array}\right)\right\|_{X} + C_\omega C'_\omega \left\|\left(\begin{array}{c}
\psi_1 \\
\psi_2
\end{array}\right)\right\|_{X}^{1/2} \left\|\A\left(\begin{array}{c}
\psi_1 \\
\psi_2
\end{array}\right)\right\|_X^{1/2}.
\end{array}$$
Therefore, inequality \eqref{ineg norme w} is satisfied for $\delta = 1/2$ and $\M_2 = \L_2$. Using \refC{Cor inv fermeture somme}, we obtain 
\begin{equation*}
D\left(\overline{\L_1 + \L_2}\right)\subset \E_{2},
\end{equation*}
which gives the expected result.
\end{proof}

\section{Back to the abstract problem}\label{Sect Back to Abstract Pb}

Now, we are in position to solve the following equation
\begin{equation}\label{eq L1+L2}
\left(\overline{\L_1 + \L_2}\right) V + k \rho^2 \left(\P_1 + \P_2 \right)V = \F.
\end{equation}
\begin{Th}\label{Th exist V strong sol}
Let $\F \in L^p(0,+\infty;X)$ and assume that \eqref{hyp inv sum} holds.  Then, there exists $\rho_0 > 0$ such that for all $\rho \in (0,\rho_0]$, equation \eqref{eq L1+L2} has a unique strong solution $V \in L^p(0,+\infty;X)$, that is
\begin{equation}\label{Def strong solution}
\left\{\begin{array}{l}
\exists\, \left(V_n\right)_{n\geqslant 0} \in D(\L_1)\cap D(\L_2) : \\ \\

V_n \underset{n \to +\infty}{\longrightarrow} V \text{ in } L^p(0,+\infty;X) \text{ and}\\ \\

\left(\L_1 + \L_2\right) V_n + k \rho^2 \left(\P_1 + \P_2 \right)V_n \underset{n \to +\infty}{\longrightarrow} \F \text{ in }L^p(0,+\infty;X),
\end{array}\right.
\end{equation}
satisfying 
\begin{equation}\label{Reg V}
V \in W^{1,p}(0,+\infty;X) \cap L^{p}\left( 0,+\infty ;\left[ W^{3,p}(0,\omega)\cap W_{0}^{2,p}(0,\omega)\right] \times L^p(0,\omega)\right).
\end{equation}
\end{Th}
\begin{proof}
Due to \refT{Th L1+L2 barre inversible}, if \eqref{hyp inv sum} holds, then $\overline{\L_1 + \L_2}$ is invertible. Thus, it follows that
\begin{equation*}
\left[ I + k \rho^2 \left(\P_1 + \P_2 \right)\left(\overline{\L_1 + \L_2}\right)^{-1} \right]\left(\overline{\L_1 + \L_2}\right) V = \F.
\end{equation*}
From \eqref{L1+L2 dans E1 inter E2}, we deduce that $V \in D(\overline{\L_1 + \L_2}) \subset \E_1 \cap \E_2$, that is \eqref{Reg V} which involves that 
$$(\P_1 + \P_2)(\overline{\L_1 + \L_2})^{-1} \in \L(X).$$ 
Then, there exists $\rho_0 > 0$ small enough such that, for all $\rho \in (0,\rho_0]$, we have
\begin{equation}\label{V}
V = \left(\overline{\L_1 + \L_2}\right)^{-1}\left[ I + k \rho^2 \left(\P_1 + \P_2 \right)\left(\overline{\L_1 + \L_2}\right)^{-1} \right]^{-1}\F,
\end{equation}
which means that $V$ is the unique strong solution of \eqref{eq L1+L2}. 
\end{proof}

\section{Proof of \refT{Th principal}}\label{Sect Proof of main Th}

From \refT{Th exist V strong sol}, there exists $\rho_0 > 0$ such that for all $\rho \in (0,\rho_0]$, equation \eqref{eq L1+L2} has a unique strong solution $V \in L^p(0,+\infty;X)$ satisfying \eqref{Reg V}. Then, due to \eqref{Def strong solution}, there exists $(V_n)_{n \in \NN} \in D(\L_1 + \L_2)$ such that $V_n \underset{n \to +\infty}{\longrightarrow} V$ and
$$\lim_{n \to +\infty} \left(\L_1 + \L_2\right) V_n + k \rho^2 \left(\P_1 + \P_2 \right)V_n = \F.$$
Since $(V_n)_{n \in \NN} \in D(\L_1 + \L_2)$, then the previous equality writes
\begin{equation}\label{Pb Vn}
\left\{\begin{array}{l}
\dis\lim_{n \to +\infty} \left(V_n''(t) - \A V_n(t) - \F_n(t)\right) = 0 \\ \ecart
\dis\lim_{n \to +\infty}V_n(0)=0, \quad \dis\lim_{n \to +\infty}V_n(+\infty) = 0,
\end{array}\right.
\end{equation}
where
$$\F_n(t) =  k\rho^{2}e^{-2t}\mathcal{A}_{0}V_n(t) + k\rho ^{2}e^{-2t}\left[ (\mathcal{B}_{2}V_n)\right] (t) + 2\nu V_n'(t) - \nu^2 V_n(t) + \F(t).$$
Since $V_n \underset{n \to +\infty}{\longrightarrow} V$ and $V$ satisfies \eqref{Reg V}, we deduce that 
$$\lim_{n \to +\infty} V_n (0) = V(0) = 0 \quad \text{with} \quad \lim_{n \to +\infty} V_n (+\infty) = V(+\infty) = 0,$$
and 
$$\lim_{n \to +\infty} \F_n (t) = \F_\infty (t) \in L^p(0,+\infty;X),$$
where
$$\F_\infty (t) =  k\rho^{2}e^{-2t} \mathcal{A}_{0}V(t) + k\rho ^{2}e^{-2t}\left[ (\mathcal{B}_{2}V)\right] (t) + 2\nu V'(t) - \nu^2 V(t) + \F(t).$$
Thus, problem \eqref{Pb Vn} writes 
\begin{equation*}
\left\{\begin{array}{l}
\dis\lim_{n \to +\infty} \left(V_n''(t) - \A V_n(t)\right) = \F_\infty (t) \\ \ecart
V(0)=0, \quad V(+\infty) = 0,
\end{array}\right.
\end{equation*}
Moreover, from \refP{Prop A BIP}, $\A \in$ BIP\,$(X,\theta_\A)$, with $\theta_\A \in (0,\pi)$ and due to \cite{haase}, Proposition 3.2.1, e), p. 71, it follows that $\sqrt{\A} \in$ BIP\,$(X,\theta_\A /2)$ with $\theta_\A /2 \in (0,\pi/2)$. Therefore, due to \cite{amine}, Theorem 2, p. 712, with $L_1 = L_2 = - \sqrt{\A}$, there exists a unique classical solution to the following problem
\begin{equation*}
\left\{\begin{array}{l}
\mathcal{V}''(t) - \A \mathcal{V}(t) = \F_\infty (t) \\ \ecart
\mathcal{V}(0)=0, \quad \mathcal{V}(+\infty) = 0,
\end{array}\right.
\end{equation*}
that is
$$\mathcal{V} \in W^{2,p}(0,+\infty;X) \cap L^p(0,+\infty;D(\A)).$$
Thus, it follows that
$$\lim_{n \to +\infty} \left(V_n''(t) - \A V_n(t)\right) = \mathcal{V}''(t) - \A \mathcal{V},$$
hence
$$\lim_{n \to +\infty} \left[ \left(V_n(t) - \mathcal{V}(t)\right)'' - \A \left(V_n(t) - \mathcal{V}(t)\right) \right] = 0.$$
Now, set
$$\left\{\begin{array}{cll}
D(\delta_2) & = & \left\{\varphi \in W^{2,p}(0,+\infty;X) : \varphi(0) = \varphi(+\infty) = 0 \right\} \\ \ecart
\delta_2 \varphi & = & \varphi'', \quad \varphi \in D(\delta_2).
\end{array}\right.$$
Then, we can write
\begin{equation}\label{lim Vn-V}
0 = \lim_{n \to +\infty} \left[\left(V_n(t) - \mathcal{V}(t)\right)'' - \A \left(V_n(t) - \mathcal{V}(t)\right) \right] = \lim_{n \to +\infty} -\left(-\delta_2 + \A\right) \left(V_n(t) - \mathcal{V}(t)\right).
\end{equation}
From \cite{pruss-sohr2}, Theorem~C, p.~166-167, it follows that $-\delta_2 \in$ BIP\,$(X,\theta_{\delta_2})$, for every $\theta_{\delta_2} \in (0,\pi)$ and due to \refP{Prop A BIP}, $\A \in$ BIP\,$(X,\theta_\A)$, for all $\theta_\A \in (0,\pi)$. Thus, since $-\delta_2$ and $\A$ are resolvent commuting with $\theta_{\delta_2} + \theta_\A < \pi$, from \cite{pruss-sohr}, Theorem 5, p. 443, we obtain that
$$-\delta_2 + \A \in \text{BIP}\,(X,\theta), \quad \theta = \max(\theta_{\delta_2},\theta_\A).$$
Moreover, due to \refP{Prop A inversible}, we have $0\in \rho(\A)$, then we deduce from \cite{pruss-sohr}, remark at the end of p. 445, that $0 \in \rho(\delta_2 + \A)$. Therefore, due to \eqref{lim Vn-V}, we obtain that
$$\lim_{n \to +\infty} V_n(t) - \mathcal{V}(t) = 0,$$
hence, since $V_n \underset{n \to +\infty}{\longrightarrow} V$, by uniqueness of the limit, we deduce that 
$$V = \mathcal{V} \in W^{2,p}(0,+\infty;X) \cap L^p(0,+\infty;D(\A)).$$
This prove that $\L_1 + \L_2$ is closed and that $V \in D(\L_1 + \L_2)$.

\end{document}